\documentclass{personalart}

\usepackage{personalart_sem}

\title{Explicit computation of the first \'etale cohomology on curves}
\author{J.~Jin \\ jjin@mpim-bonn.mpg.de \\ Max-Planck-Institut f\"ur Mathematik \\ Vivatsgasse 7 \\ 53111 Bonn \\ Germany}
\date{17 November 2017}

\begin{document}
\maketitle

\subsubsection*{Abstract}

In this paper, we describe an algorithm that, for a smooth connected curve $X$ over a field $k$ with normal completion having arithmetic genus $p_a(X)$, a finite locally constant sheaf $\sh A$ on $X_{\et}$ of abelian groups of torsion invertible in $k$, represented by a smooth curve with normal completion having arithmetic genus $p_a(\sh A)$ and degree $n$ over $X$, computes the first \'etale cohomology $\hl^1(X_{k^\sep,\et},\sh A)$ and the first \'etale cohomology with proper support $\hl^1_c(X_{k^\sep,\et},\sh A)$ as sets of torsors, in arithmetic complexity exponential in $n^{\log n}$, $p_a(X)$, and $p_a(\sh A)$.
This is done via the computation of a groupoid scheme classifying the relevant torsors (with extra rigidifying data).

\subsubsection*{Acknowledgements}

This paper is in part based on Chapter 3 of the author's dissertation, which was funded by the Netherlands Organisation for Scientific Research (project no. 613.001.110), and the author thanks his supervisors Bas Edixhoven and Lenny Taelman for their guidance during the author's PhD candidacy.
The author also thanks the Max-Planck-Institut f\"ur Mathematik in Bonn for their support during the production of this paper.

\section{Introduction}

The motivating question for this paper is the following;  this question is posed e.g.~by Poonen, Testa, and van Luijk in \cite{poonenetal15}.

\begin{qn} \label{q:main}
  Is there an algorithm that takes an algebraic variety $X$ over a field $k$, and a positive integer $n$ invertible in $k$, and computes $\Gal(k^\sep/k)$-modules isomorphic to the \'etale cohomology groups $\hl^q(X_{k^\sep,\et},\bb Z/n \bb Z)$ for $q=0,1,\dotsc,2\dim X$?
\end{qn}

We assume affine schemes of finite presentation over some base ring $R$ to be given by generators and relations over $R$, and we assume $X$ to be given by a gluing datum of affine varieties over $k$.
The output will be given as a pair $(l,X)$ of a Galois extension $l/k$ and a finite $\Gal(l/k)$-set $X$.
\cite[Th.~finitude]{sga45} guarantees that in the situation of this question, the groups $\hl^q(X_{k^\sep,\et},\bb Z/n \bb Z)$ are indeed finite.

The existence of an algorithm as in the question computing the \'etale cohomology groups in time polynomial in $n$ for a fixed variety over $\bb Q$ implies, via the Lefschetz trace formula \cite[Rapport]{sga45} and by the argument of \cite[Thm.~15.1.1]{couveignesedixhoven11}, the computation of the number of $\bb F_q$-points of some fixed finite type scheme $X$ (over $\bb Z$) in time polynomial in $\log q$.
Here, we note that the problem of computing $\#X(\bb F_q)$ has many efficient solutions in practice, see e.g.~\cite{kedlaya01}, \cite{lauderwan08}, and \cite{harvey14};  however, none of them run in time polynomial in the characteristic of the finite field.
One other application of a positive answer to the question above is the computation of N\'eron-Severi groups by Poonen, Testa, and van Luijk in \cite{poonenetal15} using the computation of the \'etale cohomology groups.

\autoref{q:main} is already known to have a positive answer.
In 2015, Poonen, Testa, and van Luijk, in the aforementioned article \cite{poonenetal15}, showed that the \'etale cohomology groups are computable if $X$ is a smooth, projective, geometrically irreducible variety over a field of characteristic $0$.
Later that year, Madore and Orgogozo \cite{madoreorgogozo15} showed that they are computable for any variety over any field, and, assuming computations with constructible sheaves can be performed, with coefficient sheaf any constructible sheaf of abelian groups (of torsion invertible in the base field).
However, both of these results are fundamentally merely computability results, without any bounds on the complexity, even for a fixed instance.

So a natural extension of \autoref{q:main} is (in addition to allowing more general coefficients) to also ask for explicit upper bounds for the complexity;  beyond the classical case of smooth curves with constant coefficients, the author doesn't know of any such result.
In this paper, we will describe an algorithm computing, for smooth connected curves, the first \'etale cohomology group (proper support or not) with coefficients in a finite locally constant sheaf of abelian groups (of torsion invertible in $k$), together with theoretical upper bounds for the complexity.

We will assume the field $k$ is given together with black box field operations (see \autoref{h:computation} for more details) and we measure the complexity only in the number of field operations performed.
While this is a good approximation for the time complexity in case $k$ is finite, for infinite $k$ this is usually not the case because of coefficient size growth.
Algorithms will be deterministic (except for the use of the black boxes);  for an actual implementation of the algorithm to be presented, it may be more efficient in practice to use randomised algorithms.
Moreover, the choice of algorithms is motivated by their theoretical worst-case complexities;  for an actual implementation, it may be significantly more efficient in practice to use different algorithms than the ones used in this paper.

With this in mind, let us state this paper's main theorem, deferring the description of the in- and output mainly to \autoref{h:standard-modules} and \autoref{h:reduction}.

\begin{thm} \label{d:main}
  There exist an algorithm that takes as input a smooth connected curve $X$ over $k$, and a (curve representing a) finite locally constant sheaf $\sh A$ of abelian groups of degree $n$ over $X$ with $n$ invertible in $k$, and return as output $\hl^1(X_{k^{\sep}},\sh A|_{X_{k^{\sep}}})$ (resp. $\hl^1_c(X_{k^{\sep}},\sh A|_{X_{k^{\sep}}})$) as $\Gal(k^{\sep}/k)$-modules in a number of field operations exponential in $n^{\log n}$, $p_a(X)$, and $p_a(\sh A)$, where $p_a$ denotes the arithmetic genus of the normal completion.
\end{thm}

More precise (and slightly more general) versions of this theorem will be given in \autoref{h:groupoid} and \autoref{h:reduction}.

\section{The geometric idea behind the algorithm}

Let $X$ be a smooth connected curve over a field $k$, and let $\sh G$ be a finite locally constant sheaf of groups.
Then the set $\hl^1(X_{k^{\sep}},\sh G|_{X_{k^{\sep}}})$, resp.~$\hl^1_c(X_{k^{\sep}},\sh G|_{X_{k^{\sep}}})$, is the set of isomorphism classes of $\sh G$-torsors on $X_{k^{\sep}}$, resp.~the set of isomorphism classes of $j_!\sh G$-torsors on $\nc X_{k^{\sep}}$.
Here, $j \colon X \to \nc{X}$ is the open immersion of $X$ into its normal completion $\nc{X}$.

\begin{rem}
  If $\sh G$ is a sheaf of abelian groups, a priori we have two possible definitions of $j_!\sh G$;  one arising from viewing $j_!$ as the left adjoint of $j_*$ on the category of sheaves of groups, and one arising from viewing $j_!$ as the left adjoint of $j_*$ on the category of sheaves of abelian groups.
  Let us call these $j^G_! \sh G$ and $j^A_! \sh G$ for now.
  There is a natural map $j^G_! \sh G \to j^A_! \sh G$, which induces an isomorphism on stalks since $j$ is an open immersion (and direct sums of zero, resp.~one object in the category of groups and that of abelian groups have the same underlying sets).
  Hence $j^G_! \sh G = j^A_! \sh G$, so there is no confusion possible if we just write $j_! \sh G$ in this case, like we did above.
\end{rem}

Now, in short, the idea behind the algorithm is to compute a moduli scheme of torsors with some additional structure (which is to be specified) and use this moduli scheme to compute the first cohomology.
The moduli space in question is constructed in \autoref{h:groupoid}.

We describe this idea in some more detail below, using the language of stacks.

Let $\stk T$ be the stack on $(\Sch/X)_{\fppf}$ of $\sh G$-torsors, resp.~the stack on $(\Sch/\nc{X})_{\fppf}$ of $j_!\sh G$-torsors;  note that this stack is presented by the group(oid) scheme $\sh G \to X$, resp.~$j_!\sh G \to \nc{X}$.
Let $f$ be the structure morphism $X \to \Spec k$, resp.~the structure morphism $\nc{X} \to \Spec k$, and let $p$ denote the morphism from the big \'etale topos to the small \'etale topos for which $p_*$ is the restriction to the small site.
Note that for this $p$, the functor $p^{-1}$ is the espace \'etal\'e functor.
Let $f_{\sbig,*}$ and $f_{\ssmall,*}$ denote the big and small pushforward, respectively.

We then have a stack $f_{\ssmall,*}p_* \stk T = p_*f_{\sbig,*} \stk T$ on $(\Spec k)_{\et}$, to which we can attach the sheaf $\pi_0(f_{\ssmall,*}p_* \stk T)$ on $(\Spec k)_{\et}$, and a morphism $f_{\ssmall,*}p_* \stk T \to \pi_0(f_{\ssmall,*}p_* \stk T)$ of stacks on $(\Spec k)_{\et}$.
The Galois set to be computed now corresponds to the sheaf $\pi_0(f_{\ssmall,*}p_* \stk T)$ on $(\Spec k)_{\et}$, or in other words, to the \'etale $k$-scheme $p^{-1} \pi_0(f_{\ssmall,*}p_* \stk T)$.

We show in \autoref{h:torsors} that the diagonal of $f_{\sbig,*}\stk T$ is representable and finite \'etale, which simply means that for any $k$-scheme $S$ and any two objects $X,Y$ of $f_{\sbig,*} \stk T(S)$, the corresponding sheaf $\Hom_{f_{\sbig,*} \stk T(S)}(X,Y)$ on $(\Sch/k)_{\et}$ is representable by a finite \'etale $S$-scheme.

In \autoref{h:groupoid} we compute a groupoid scheme $\mdl R \rightrightarrows \mdl U$ with $\mdl R$ and $\mdl U$ affine schemes of finite type over $k$, together with an obvious (non-explicit) morphism $[\mdl U/\mdl R] \to f_{\sbig,*} \stk T$ of stacks on $(\Sch/k)_{\et}$.
There, we also show that $p^{-1}p_*[\mdl U/\mdl R] \to p^{-1}f_{\ssmall,*}p_* \stk T$ is an equivalence (after some purely inseparable base change, but we ignore this technical point for now), and that the morphisms $\mdl R \rightrightarrows \mdl U$ are smooth and have geometrically irreducible fibres.

Hence we are (after some purely inseparable base change) in the situation of the following proposition.

\begin{prop} \label{d:stack-key-lemma}
  Let $\stk T$ be a stack on $(\Sch/k)_{\fppf}$ of which the diagonal is representable and finite \'etale.
  Let $\mdl R \rightrightarrows \mdl U$ be a groupoid scheme such that both morphisms $\mdl R \to \mdl U$ are smooth and have geometrically connected fibres, and such that $\mdl R$ and $\mdl U$ are of finite type over $k$.
  Let $[\mdl U/\mdl R] \to \stk T$ be a morphism of stacks on $(\Sch/k)_{\fppf}$ such that $p^{-1}p_*[\mdl U/\mdl R] \to p^{-1}p_* \stk T$ is an equivalence, or in other words, such that for each separable extension $l/k$, the functor $[\mdl U/\mdl R](l) \to \stk T(l)$ is an equivalence.
  Then the map $\mdl U(k^{\sep}) \to \pi_0(\stk T)(k^{\sep})$ is a $\Gal(k^{\sep}/k)$-equivariant surjection, and factors through an isomorphism $\pi_0(\mdl U) \to \pi_0(\stk T)$.

  If in addition the morphisms $\mdl R \to \mdl U$ have geometrically irreducible fibres, then the connected components of $\mdl U_{k^{\sep}}$ are irreducible.
\end{prop}

\begin{proof}
  First note that the equivalence $p^{-1}p_*[\mdl U/\mdl R] \to p^{-1}p^* \stk T$ induces a surjective $\Gal(k^{\sep}/k)$-equivariant map $U(k^{\sep}) \to \pi_0(\stk T)(k^{\sep})$.

  Let $\gpt x \in \mdl U(k^{\sep})$ and let $j \colon U \to \mdl U_{k^{\sep}}$ be the open immersion of the connected component $U$ containing $\gpt x$ into $\mdl U_{k^{\sep}}$.
  Moreover, let $f \colon U \to \Spec k^{\sep}$ denote the structure morphism, and let $p \colon \mdl U_{k^{\sep}} \to \mdl U$ be the projection morphism.
  Let $\uni Y \in \mdl U(\mdl U)$ denote the ``universal object''; i.e.~the object corresponding to the identity map on $\mdl U$.

  Define $Y_1 = j^{-1}p^{-1} \uni Y, Y_2 = f^{-1}\gpt x^{-1} \uni Y \in \mdl U(U)$, and consider their images in $\Ob \stk T(U)$.
  Then $\Isom_{\stk T(U)}(Y_1,Y_2)$ is representable by a finite \'etale $U$-scheme by assumption.

  Moreover, it is surjective as by construction $\Hom_{\stk T(U)}(Y_1,Y_2)(\gpt x)$ is non-empty and $U$ is connected.
  Hence for any point $\gpt{x'} \in \mdl U(k^{\sep})$, the set $\Hom_{\stk T(U)}(Y_1,Y_2)(\gpt{x'})$ is non-empty as well.
  Therefore the morphism $\mdl U(k^{\sep}) \to \pi_0(\stk T)(k^{\sep})$ factors through a surjective $\Gal(k^{\sep}/k)$-equivariant map $\pi_0(\mdl U)(k^{\sep}) \to \pi_0(\stk T)(k^{\sep})$.
  In other words, the morphism $\pi_0(\mdl U) \to \pi_0(\stk T)$ of sheaves on $(\Spec k)_{\et}$ is surjective.

  Denote the morphisms $\mdl R \to \mdl U$ by $\alpha$ and $\omega$.
  As $\alpha$ and $\omega$ have geometrically connected fibres, it follows that the morphism $\pi_0(\mdl U) \to \pi_0(\stk T)$ is an isomorphism; if $\gpt x, \gpt x' \in \mdl U(k^{\sep})$ are isomorphic, then $\gpt x' \in \alpha\pth*[big]{\omega^{-1}(\gpt x)}$, hence $\gpt x, \gpt x'$ lie in the same geometric connected component of $\mdl U$.

  Finally, if $\alpha$ and $\omega$ have geometrically irreducible fibres, then the same argument implies that every geometric connected component of $\mdl U$ is irreducible.
\end{proof}

As a corollary, we see that in \autoref{d:stack-key-lemma} any set $X$ of points $\Spec l_i \to \mdl U$ (with $l_i/k$ finite algebraic) such that every connected component of $\mdl U$ contains a point $X$ induces a finite cover of $\pi_0(\stk T)$.
We describe in \autoref{h:points-torsors} how to use such a set $X$ to compute $\pi_0(\stk T)$.

\section{Preliminaries} \label{h:computation}

In this paper we will mainly consider {\em generic field algorithms}; i.e.~algorithms that take a finite number of bits and a finite number of a field $k$, which are only allowed to operate on the field elements through a number of black box operations, and, aside from the black box operations, are deterministic.
The assumptions that follow here are essentially the assumptions as mentioned in \cite[p.~1843, 1846]{chistov86}.

First, we assume the constants $0$ and $1$ (in $k$) and the characteristic exponent $p$ (in $\bb Z$) are given.
Moreover, we assume the imperfectness degree $p^e$ of $k$ to be finite, and that $k^{1/p}/k$ is given explicitly as a finite $k$-algebra (i.e.~as a $k$-vector space with given unit and multiplication table).

In this paper these black box operations are:
\begin{itemize}
  \item $=$, which takes $x,y \in k$, and returns $1$ if $x=y$, and $0$ if $x \neq y$;
  \item $-$, which takes $x \in k$, and returns $-x$;
  \item $\cdot^{-1}$, which takes $x \in k$, and returns nothing if $x=0$, and $x^{-1}$ if $x \neq 0$;
  \item $+$, which takes $x,y \in k$, and returns $x+y$;
  \item $\times$, which takes $x,y \in k$, and returns $xy$;
  \item $\cdot^{1/p}$, which takes $x \in k$, and returns $x^{1/p} \in k^{1/p}$;
  \item $F$, which takes a polynomial $f \in k[x]$, and returns its factorisation into irreducibles in $k[x]$.
\end{itemize}

\begin{rem}
  For any field finitely generated over a finite field or $\bb Q$, there are algorithms for each of the above black box operations, however, the most efficient implementations of the factorisation algorithm for finite fields are randomised.
\end{rem}

To such a generic field algorithm we attach a number of functions (from the set of inputs to $\bb N$).
\begin{itemize}
  \item The {\em bit-complexity} $N_{\cbit}$; for an input $I$, the number $N_{\cbit}(I)$ is the number of bit-operations the algorithm performs when given $I$.
  \item The {\em arithmetic complexity} $N_{\car}$; for an input $I$, the number $N_{\car}(I)$ is the number of black box operations the algorithm performs when given $I$.
\end{itemize}
We will usually not mention the bit-complexity of the algorithms in this paper;  in all cases, the bit-complexity will be small compared to the arithmetic complexity.
As is customary, as a measure of size for inputs, we take the pair $(b,f)$, where $b$ is the number of bits in the input, and $f$ is the number of field elements in the input;  so for $\Phi$ a function from the set of inputs to $\bb N$, we will denote by $\Phi(b,f)$ the maximum of the $\Phi(I)$ with $I$ ranging over all the inputs with at most $b$ bits and $f$ field elements.

We note that a lot of linear algebraic operations, like matrix addition, matrix multiplication, computation of characteristic polynomial, and by extension, reduced row echelon form, rank, kernels, images, quotients, etc.~can all be performed in arithmetic complexity polynomial in the size of the input.

By \cite[Sec.~7]{khurimakdisi04}, the primary decomposition of a finite $k$-algebra $A$ can also be computed in arithmetic complexity polynomial in $[A:k]$, and if $k$ is perfect, the same holds for the computation of nilradicals.
In fact, in our case \cite[Sec.~7]{khurimakdisi04} computes an $l$-basis (and therefore a $k$-basis) for the nilradical of $A \otimes_k l$ (where $l = k^{1/p^{\flr*{\log_p [A:k]}}}$), and therefore also a $k$-basis for the nilradical of $A$, in arithmetic complexity polynomial in $[A:k]^{e+1}$.

Moreover, using the criteria that a reduced finite $k$-algebra $A$ is separable if $[A:k] < p$, and if and only if $A$ is spanned over $k$ by $t_i^p$ for $t_i$ any $k$-basis for $A$, one can compute separable closures of $k$ in finite field extensions $l$ in arithmetic complexity polynomial in $[l:k]$, using the obvious recursive algorithm.

By \cite[Sec.~1.1]{chistov86} we have algorithms which compute for a finite field extension $l/k$ the extension $l^{1/p}/l$ and the operations listed above; aside from the computation of $l^{1/p}/l$, that of characteristic roots, which have arithmetic complexity polynomial in $[l:k]^{e+1}$, and that of factorisation, which has arithmetic complexity polynomial in $[l:k]^{e+1}$ and the degree of the polynomial to be factored, every operation has arithmetic complexity polynomial in $[l:k]$.
Moreover, $l$ has the same characteristic exponent and imperfectness degree as $k$.

Now consider the purely transcendental extension $k(x)/k$.
We present its elements by pairs of polynomials;  for $f,g \in k[x]$ we set the {\em height} of $\frac fg$ to be $h(\frac fg) = \max(\deg f,\deg g)$.
Then note that for $k(x)/k$, we have $k(x)^{1/p} = k^{1/p}(x^{1/p})$ and therefore an obvious $k(x)$-basis for $k(x)^{1/p}$, and we can compute the listed operations for elements of $k(x)$ of height at most $h$ in arithmetic complexity polynomial in $h(x)$.
(Again, with the exceptions of characteristic roots, which has arithmetic complexity polynomial in $h(x)^{e+2}$ and polynomial factorisation, which has arithmetic complexity polynomial in $h(x)^{e+2}$ and the degree of the polynomial to be factored, see e.g.~\cite{kaltofen85}.)

As is customary, we will use the standard big-oh notation when bounding complexities; moreover, we will use $O(x,y)$ as a shorthand for $O\pth*[big]{\max(x,y)}$.

\section{Parametrising morphisms of modules} \label{h:standard-modules}

We will use the following characterisation of isomorphism classes of vector bundles over $\bb P^1_k$ with $k$ a field.

\begin{prop}[{\citet{dedekindweber1882}}]
  Let $k$ be a field, and let $\sh E$ be a vector bundle on $\bb P^1_k$.
  Then there exists an up to permutation unique finite sequence $(a_i)_{i=1}^s$ of integers such that
  \[
    \sh E \iso \bigoplus_{i=1}^s \sh O_{\bb P^1_k}(a_i).
  \]
\end{prop}

This motivates the following definition.

\begin{defn}
  Let $S$ be a scheme, and let $a$ be a finite sequence of integers of length $s$.
  The {\em standard module of type $a$} on $S$ is the $\sh O_{\bb P^1_S}$-module
  \[
    \sh O_{\bb P^1_S}(a) = \bigoplus_{i=1}^s \sh O_{\bb P^1_S}(a_i).
  \]
\end{defn}

So every vector bundle $\sh E$ over $\bb P^1_k$ is isomorphic to a standard module over $k$, say of type $a$; in this case, we simply say that $\sh E$ {\em has type} $a$.

Let, for finite sequences $a,b$ of integers, of lengths $s,t$, respectively, $H_{a,b}$ define the functor $\Sch^{\op} \to \Set$ sending $S$ to $\Hom_{\sh O_{\bb P^1_S}}\pth*[big]{\sh O_{\bb P^1_S}(b),\sh O_{\bb P^1_S}(a)}$.
Moreover, let $N(a,b) = \sum_{i=1,j=1}^{s,t} \max(a_i-b_j+1,0)$.

Then the functor $H_{a,b}$ is representable by $\bb A^{N(a,b)}$:  in fact, as
\[
  \Hom_{\sh O_{\bb P^1_S}}\pth*[big]{\sh O_{\bb P^1_S}(b_j),\sh O_{\bb P^1_S}(a_i)} = \sh O(S)[x,y]_{a_i-b_j}
\]
functorial in $S$, we get an identification
\[
  \Hom_{\sh O_{\bb P^1_S}}\pth*[big]{\sh O_{\bb P^1_S}(b),\sh O_{\bb P^1_S}(a)} = \set*[Big]{M \in \Mat_{s \times t}\pth*[big]{\sh O(S)[x,y]} : M_{ij} \in \sh O(S)[x,y]_{a_i - b_j}},
\]
and under this identification, all the relevant operations on morphisms of standard modules (i.e.~identity map, composition, direct sum, tensor product, dual, exterior powers) correspond to their usual counterparts on matrices.
In particular, if these operations are viewed as operations on the representing scheme $\bb A^{N(a,b)}$, then the degrees of the polynomials defining them are as expected.

To an element of $H_{a,b}(S)$, one way to give its fibre at $0 \in \bb P^1_S$ is by substituting $(0,1)$ for $(x,y)$, and one way to give its fibre at $\infty \in \bb P^1_S$ is by substituting $(1,0)$ for $(x,y)$.
Moreover, a way to give the first infinitesimal neighbourhood at $\infty \in \bb P^1_S$ is by substituting $x = 1$ and setting $y^2 = 0$.

\section{Torsors} \label{h:torsors}

A {\em curve} over a field $k$ in this paper is a separated $k$-scheme of finite type, of pure dimension $1$ over $k$.

Consider the following situation.

Let $k$ be a field, let $f \colon X \to \Spec k$ be a smooth connected curve, let $i \colon Z \to X$ be a closed immersion, let $j \colon U \to X$ be its open complement, and let $\sh G$ be a finite locally constant sheaf of groups on $U_{\et}$.
We wish to compute $R^1f_*j_! \sh G$ (or equivalently, $\hl^1(X_{k^{\sep},\et},j_!\sh G)$ as a Galois set) under some minor conditions, and as stated before, we will do this by reduction to a computation with standard modules.

\subsection{$j_!\sh G$-torsors and recollement}

We first recall {\em recollement}.

\begin{defn}
  Let $X$ be a scheme, let $i \colon Z \to X$ be a closed immersion, and let $j \colon U \to X$ be its open complement.

  Define the category $\Sh_{Z,U}(X_{\et})$ as follows.
  The set of objects of $\Sh_{Z,U}(X_{\et})$ is the set of triples $(\sh F_Z,\sh F_U,\phi)$ of a sheaf $\sh F_Z$ on $Z_{\et}$, a sheaf $\sh F_U$ on $U_{\et}$, and a morphism $\phi \colon \sh F_Z \to i^{-1}j_* \sh F_U$.
  For objects $(\sh F_Z,\sh F_U,\phi)$ and $(\sh F'_Z,\sh F'_U,\phi')$ of $\Sh_{Z,U}(X_{\et})$, the set of morphisms from $(\sh F_Z,\sh F_U,\phi)$ to $(\sh F'_Z,\sh F'_U,\phi')$ is the set of pairs $(f_Z,f_U)$ of a morphism $f_Z \colon \sh F_Z \to \sh F'_Z$ and a morphism $f_U \colon \sh F_U \to \sh F'_U$ such that the following diagram commutes.
  \[
    \begin{tikzcd}
      \sh F_Z \ar[d,"\phi"'] \ar[r,"f_Z"] & \sh F'_Z \ar[d,"\phi'"] \\
      i^{-1}j_* \sh F_U \ar[r,"i^{-1}j_*{(f_U)}"'] & i^{-1}j_* \sh F'_U
    \end{tikzcd}
  \]
\end{defn}

\begin{thm}[Recollement, e.g.~{\cite[Sec.~5.4]{fu15}}]
  Let $X$ be a scheme, let $i \colon Z \to X$ be a closed immersion, and let $j \colon U \to X$ be its open complement.

  Then the functor $\Sh(X_{\et}) \to \Sh_{Z,U}(X_{\et})$ sending $\sh F$ to $\pth*[big]{i^{-1}\sh F,j^{-1}\sh F,i^{-1}(\upsilon)}$, where $\upsilon \colon \sh F \to j_*j^{-1}\sh F$ is the unit map of the adjoint pair $(j^{-1},j_*)$ of functors, is an equivalence of categories, and a quasi-inverse $\Sh_{Z,U}(X_{\et}) \to \Sh(X_{\et})$ is given by sending $(\sh F_Z,\sh F_U,\phi)$ to $i_*\sh F_Z \times_{i_*(\phi),i_*i^{-1}j_* \sh F_U,\upsilon} j_* \sh F_U$, where $\upsilon \colon j_* \sh F_U \to i_* i^{-1} j_* \sh F_U$ is the unit map of the adjoint pair $(i^{-1},i_*)$ of functors.
\end{thm}

Note that the functor $i^{-1}j_*$ is left exact, hence commutes with finite limits.
Let $\cat T$ denote the category of $j_!\sh G$-torsors on $X_{\et}$, and let $\cat T_{Z,U}$ denote the category of which the objects are pairs $(\sh F,s)$ of a $\sh G$-torsor $\sh F$ on $U_{\et}$, and a section $s \in i^{-1}j_*\sh F(Z)$, and in which the morphisms $(\sh F,s) \to (\sh F',s')$ are the morphisms $f \colon \sh F \to \sh F'$ such that $i^{-1}j_*(f)$ sends $s$ to $s'$.

\begin{lem}
  Let $X$ be a scheme, let $i \colon Z \to X$ be a closed immersion, and let $j \colon U \to X$ be its open complement.
  Let $\sh G$ be a sheaf of groups on $U_{\et}$.
  The rule attaching to a $j_!\sh G$-torsor $\sh F$ on $X_{\et}$ the pair $\pth*[big]{j^{-1}\sh F,i^{-1}(\upsilon)}$, where $\upsilon \colon \sh F \to j_*j^{-1}\sh F$ denotes the unit map of the adjoint pair $(j^{-1},j_*)$ of functors, defines an equivalence $\cat T \to \cat T_{Z,U}$ of categories.
\end{lem}

\begin{proof}
  First, note that the sheaf $j_! \sh G$ is under recollement equivalent to the triple $(1,\sh G,1)$.

  Now giving a $j_!\sh G$-action $\rho$ on a sheaf $\sh F$ on $X_{\et}$ is equivalent to giving $(i^{-1} \sh F,j^{-1} \sh F,i^{-1}(\upsilon))$ together with an action of $\sh G$ on $j^{-1} \sh F$; the commutativity of
  \[
    \begin{tikzcd}
      1 \times i^{-1} \sh F \ar[d,"1 \times {i^{-1}(\upsilon)}"'] \ar[r,"\rho_Z"] & i^{-1}\sh F \ar[d,"{i^{-1}(\upsilon)}"] \\
      i^{-1}j_* \sh G \times i^{-1}j_*j^{-1} \sh F \ar[r,"i^{-1}j_*{(\rho_U)}"'] & i^{-1}j_*j^{-1} \sh F
    \end{tikzcd}
  \]
  is automatic since both $\rho_Z$ and $i^{-1}j_*(\rho_U)$ are group actions.
  (Of course, one can also deduce this equivalence by noting that a morphism $j_!\sh G \to \sAut(\sh F)$ is equivalent to a morphism $\sh G \to j^{-1} \sAut(\sh F) = \sAut(\sh F_U)$.)

  Now $\sh F$ is a $j_!\sh G$-torsor if and only if the map $j_!\sh G \times \sh F \to \sh F \times \sh F$ given on local sections by $(g,s) \mapsto (s,gs)$ is an isomorphism, and $\sh F$ locally has a section.
  This is equivalent to the following.
  \begin{itemize}
    \item $i^{-1} \sh F$ is the terminal sheaf on $Z_{\et}$; therefore $i^{-1}(\upsilon)$ is an element of $i^{-1}j_* \sh F(Z)$, and it follows that the given rule indeed defines a functor;
    \item $j^{-1} \sh F$ is a $\sh G$-torsor on $U_{\et}$,
  \end{itemize}
  so the given rule defines an equivalence, as desired.
\end{proof}

\subsection{Pushforward and normalisation}

Next, we consider a description of the pushforward of a finite locally constant sheaf along certain open immersions.
This is mostly well-known, but the author doesn't know of a reference, so proofs are included here for completeness.

Recall that, for a scheme $X$, the category of sheaves on $X_{\et}$ is equivalent to that of algebraic spaces \'etale over $X$.
Quasi-inverses are given by the functor sending an algebraic space \'etale over $X$ to its functor of points, and the functor sending a sheaf on $X_{\et}$ to its {\em espace \'etal\'e}.
By descent, finite locally constant sheaves on $X_{\et}$ are precisely those of which the espace \'etal\'e is a finite \'etale $X$-scheme.

\begin{prop} \label{d:pushforward-normalisation}
  Let $X$ be a scheme, and let $j \colon U \to X$ be a quasi-compact open immersion such that the normalisation of $X$ in $U$ is $X$.
  Let $\sh F$ be a finite locally constant sheaf on $U_{\et}$, or equivalently, a finite \'etale $U$-scheme.
  Let $\nc{\sh F}$ be the normalisation of $X$ in $\sh F$.
  Then for all \'etale $X$-schemes $T$, we have $j_*\sh F(T) = \nc{\sh F}(T)$ functorial in $T$.
\end{prop}

\begin{proof}
  First of all, note that we may restrict ourself to \'etale $X$-schemes $T$ that are affine, and therefore to quasi-compact separated \'etale $X$-schemes $T$.
  So let $T$ be an \'etale quasi-compact separated $X$-scheme.
  Let $\nc T$ be the normalisation of $X$ in $T$, and let $\nc{U \times_X T}$ be the normalisation of $X$ in $U \times_X T$.
  Then $j_*\sh F(T) = \sh F(U \times_X T)$, and we have a map $\sh F(U \times_X T) \to \nc{\sh F}(\nc{U \times_X T})$.
  Since for every $Y$-morphism $\nc{U \times_X T} \to \nc{\sh F}$, the composition with $U \times_X T \to \nc{U \times_X T}$ factors through $\sh F$ (as $\sh F$ is a finite \'etale $X$-scheme), it follows that $\sh F(U \times_X T) = \nc{\sh F}(\nc{U \times_X T})$.

  Now note that since normalisation commutes with smooth base change (see e.g.~\cite[Tag 082F]{stacksproject}), it follows that the normalisation of $T$ in $U \times_X T$ is simply $T$.
  Therefore $\nc{U \times_X T} \to \nc T$ is an isomorphism, and we have $\nc{\sh F}(\nc{U \times_X T}) = \nc{\sh F}(\nc{T}) = \nc{\sh F}(T)$, as desired.
\end{proof}

\begin{cor}
  Let $X$ be a scheme, and let $j \colon U \to X$ be a quasi-compact open immersion such that the normalisation of $X$ in $U$ is $X$.
  Then for all finite sets $F$, we have $j_* F = F$.
\end{cor}

\begin{lem}
  Let $k$ be a field, let $X$ be a $k$-scheme of finite type, and let $j \colon U \to X$ be an open immersion such that the normalisation of $X$ in $U$ is $X$.
  Let $\sh F$ be a finite locally constant sheaf on $U_{\et}$, or equivalently, a finite \'etale $U$-scheme.
  Then $j_* \sh F$ is representable by an \'etale, quasi-compact, separated $X$-scheme.
\end{lem}

\begin{proof}
  First note that by \cite[Th.~finitude]{sga45} $j_* \sh F$ is constructible, i.e.~of finite presentation as an $X$-space.

  Note that $\sh F$ is finite locally constant, so $\sh F \times \sh F$ is the disjoint union of the diagonal and its complement, inducing a morphism $\sh F \times \sh F \to \bb Z/2 \bb Z$ such that the equaliser with the constant map with value $0$ is the diagonal.
  Applying the left exact functor $j_*$ to this gives a morphism $j_* \sh F \times j_* \sh F \to j_* (\bb Z/2 \bb Z) = \bb Z/2 \bb Z$ such that the equaliser with the constant map with value $0$ is the diagonal.
  Therefore $j_* \sh F$ is separated as an $X$-space.

  It follows by \cite[Tag 03XX]{stacksproject} that $j_* \sh F$ is representable by an \'etale, quasi-compact, separated $X$-scheme.
\end{proof}

\begin{lem}
  Let $k$ be a field, let $X$ be a $k$-scheme of finite type, and let $j \colon U \to X$ be an open immersion such that the normalisation of $X$ in $U$ is $X$.
  Let $\sh F$ be a finite locally constant sheaf on $U_{\et}$, or equivalently, a finite \'etale $U$-scheme.
  Let $\nc{\sh F}$ be the normalisation of $X$ in $\sh F$.
  Then $\nc{\sh F}$ is the normalisation of $X$ in $j_* \sh F$.
\end{lem}

\begin{proof}
  First note that we have a canonical morphism $j_*\sh F \to \nc{\sh F}$ corresponding to the identity section of $j_*\sh F$.
  Let $j_*\sh F \to Y \to X$ be a factorisation with $Y \to X$ integral.
  As $\nc{\sh F}$ is the normalisation of $X$ in $\sh F$, it follows that there exists a unique morphism $\nc{\sh F} \to Y$ such that the diagram
  \[
    \begin{tikzcd}
      \sh F \ar[d] \ar[r] & \nc{\sh F} \ar[d] \ar[dl] \\
      Y \ar[r] & X
    \end{tikzcd}
  \]
  commutes.
  We show that this morphism also makes the diagram
  \begin{equation} \label{e:pushforward-normalisation}
    \begin{tikzcd}
      j_* \sh F \ar[r] \ar[d] & \nc{\sh F} \ar[dl] \\
      Y
    \end{tikzcd}
  \end{equation}
  commute.
  Let $T$ be an \'etale quasi-compact separated $X$-scheme, and consider the following diagram.
  \[
    \begin{tikzcd}
      j_* \sh F(T) \ar[r] \ar[d] & \nc{\sh F}(T) \ar[dl] \\
      Y(T)
    \end{tikzcd}
  \]
  As the normalisations of $X$ in $T$ and $U \times_X T$ are equal, as in the proof of \autoref{d:pushforward-normalisation}, the commutativity of this diagram is equivalent to the commutativity of the following one.
  \[
    \begin{tikzcd}
      \sh F(U \times_X T) \ar[r] \ar[d] & \nc{\sh F}(U \times_X T) \ar[dl] \\
      Y(U \times_X T)
    \end{tikzcd}
  \]
  It follows that the commutativity of \eqref{e:pushforward-normalisation} holds when restricted to $X_{\et}$.

  Therefore, applying this to the identity section on $j_* \sh F$, it follows that \eqref{e:pushforward-normalisation} itself commutes.
\end{proof}

By Zariski's Main Theorem, we have the following.

\begin{cor}
  The canonical morphism $j_*\sh F \to \nc{\sh F}$ is an open immersion identifying $j_* \sh F$ with the \'etale locus of $\nc{\sh F}$ over $X$.
\end{cor}

\begin{proof}
  The \'etale locus $V$ of $\nc{\sh F}$ over $X$ is open in $\nc{\sh F}$ and \'etale over $X$, therefore factors through $j_* \sh F$.
  By maximality of $V$ we get $j_* \sh F = V$.
\end{proof}

\subsection{Galois actions on finite locally constant sheaves}

Let $X$ be a scheme, and let $\Gamma$ be a group acting on $X$.
Then recall that a {\em $\Gamma$-sheaf} on $X_{\et}$ is a sheaf $\sh F$ on $X_{\et}$ of which the espace \'etal\'e is a $\Gamma$-equivariant $X$-space.

Let $X$ be a connected scheme, let $\Gamma$ be a finite group, and let $f \colon Y \to X$ be a finite \'etale connected Galois cover with Galois group $\Gamma$.
Note that pullback of sheaves defines an equivalence from the category of sheaves on $X_{\et}$ to that of $\Gamma$-sheaves on $Y_{\et}$.
A quasi-inverse is given in terms of sheaves by sending $\sh F$ to the sheaf of $\Gamma$-invariants of $f_* \sh F$; in terms of espaces \'etal\'es, it sends an algebraic space $Z$ \'etale over $Y$ to the quotient $\Gamma \backslash Z$.

If $\sh G$ is a finite locally constant sheaf of groups on $X_{\et}$ such that $f^{-1} \sh G$ is constant, let $G$ be the group of connected components of $f^{-1} \sh G$, and note that $\Gamma$ acts on $G$ by automorphisms.
Therefore we see that a finite locally constant sheaf on $X_{\et}$ with $\sh G$-action corresponds to a $\Gamma$- and $G$-equivariant finite \'etale $Y$-scheme.

Let us now apply this to the following situation.

\begin{situ} \label{s:galois-general}
  Let $k$ be a field.
  Suppose we have a finite group $\Gamma$, and a diagram of schemes of finite type over $k$
  \[
    \begin{tikzcd}
      V \ar[d,"g"'] \ar[r,"j'"] & Y \ar[d,"f"] & W \ar[d,"h"] \ar[l,"i'"'] \\
      U \ar[r,"j"'] & X & Z \ar[l,"i"]
    \end{tikzcd}
  \]
  where $U$ and $X$ are connected, $g$ is finite \'etale Galois with Galois group $\Gamma$, $Y$ is the normalisation of $X$ in $V$, $W = Y \times_X Z$, and $j$ is the open complement of $i$.
  Let $\sh G$ be a finite locally constant sheaf of groups on $\sh U$ such that $g^{-1} \sh G$ is constant, say with group of connected components $G$.
\end{situ}

Let $\cat T_{Z,U}$ be as in the previous section, and let $\cat T_{W,Y}^\Gamma$ be the category of which the objects are pairs $(\sh F,s)$ of a $\Gamma$-equivariant $G$-torsor $\sh F$ on $Y_{\et}$, and a $\Gamma$-equivariant section $s \in (i')^{-1}\sh F(W)$, and in which the morphisms $(\sh F,s) \to (\sh F',s')$ are the $\Gamma$-equivariant morphisms $f \colon \sh F \to \sh F'$ such that $(i')^{-1}(f)$ sends $s$ to $s'$.

\begin{lem} \label{d:galois-general}
  In \autoref{s:galois-general}, the rule attaching to a pair $(\sh F,s)$ of a $\sh G$-torsor $\sh F$ and a section $s \in i^{-1}j_*\sh F(Z)$ the pair $(j'_*g^{-1}\sh F,s)$ defines an equivalence $\cat T_{Z,U} \to \cat T_{W,Y}^\Gamma$ of categories.
\end{lem}

\begin{proof}
  First note that giving a $\sh G$-torsor $\sh F$ on $U_{\et}$ is equivalent to giving the $\Gamma$-equivariant $G$-torsor $g^{-1}\sh F$ on $V_{\et}$.
  Moreover, giving the section $s \colon Z \to i^{-1}j_* \sh F$ is the same as giving a $\Gamma$-invariant section $Z \to i^{-1}j_* g_* g^{-1} \sh F = i^{-1}f_*j'_*g^{-1} \sh F = h_* (i')^{-1} j'_* g^{-1} \sh F$, where the last step uses proper base change.
  This is the same as giving a $\Gamma$-equivariant section $W \to (i')^{-1} j'_* g^{-1} \sh F$.

  Now $j'_*g^{-1}\sh F$ is a $\Gamma$-equivariant $G$-pseudotorsor which \'etale locally has a section, i.e.~a $\Gamma$-equivariant $G$-torsor.
  Therefore giving the pair $(g^{-1}\sh F,s)$ is equivalent to giving $(j'_*g^{-1}\sh F,s)$, as desired.
\end{proof}
 
\subsection{Stacks of torsors}

Let $S$ be a scheme.
A {\em topologically finite \'etale} $S$-scheme $T$ is a morphism $T \to S$ that factors as a composition $T \to T' \to S$ with $T' \to S$ finite \'etale and $T \to T'$ a universal homeomorphism.

Let $f \colon X \to S$ be a proper smooth curve, let $i \colon Y \to X$ be a closed immersion, topologically finite \'etale over $S$, and let $j \colon U \to X$ denote its open complement.
Write $h = fi$ and $g = fj$.
Let $G$ be a finite group; if $Y$ is non-empty, we also assume that the order of $G$ is invertible on $S$.

Let $\cat T$ denote the fppf stack of $G$-torsors on $U_{\et}$; i.e.~its objects are pairs $(T,\sh F)$ of an $S$-scheme $T$ and a $G$-torsor $\sh F$ on $(U \times_S T)_{\et}$, and the morphisms $(T,\sh F) \to (T', \sh F')$ are the pairs of a morphism $\phi \colon T \to T'$ and an isomorphism $\phi^{-1} \sh F' \to \sh F$.
We show that $\cat T$ has a representable and finite \'etale diagonal, or equivalently, the relevant Isom-sheaves are representable by finite \'etale schemes.

Without loss of generality, and to ease notation a bit, we will only consider the Isom-sheaves on $S$.
More precisely, let $\sh F$ and $\sh F'$ be $G$-torsors on $U_{\et}$, and let $\sh I$ denote the sheaf on $\site{U}_{\fppf}$ sending $\phi \colon T \to U$ to the set $\Isom_T(\phi^{-1}\sh F,\phi^{-1}F')$ of isomorphisms of $G$-torsors.
We denote by $g_{\bsit,*}$ the big pushforward functor $\site U_{\fppf} \to \site S_{\fppf}$.

\begin{lem}
  The sheaf $g_{\bsit,*} \sh I$ on $\site S_{\fppf}$ is representable by a finite \'etale $S$-scheme.
\end{lem}

\begin{proof}
  By the theory of Hilbert schemes, the fppf sheaf $g_{\text{big},*} \sh I$ is representable by an algebraic space, locally of finite presentation over $S$, since both $\sh F$ and $\sh F'$ are representable by finite \'etale algebraic spaces over $U$.
  Moreover, we easily see that it is formally \'etale over $S$, therefore \'etale over $S$.
  Hence it is representable by the espace \'etal\'e of $(g_{\text{big},*} \sh I)|_{S_{\et}} = g_*(\sh I|_{U_{\et}})$.

  Now we note that $\sh I|_{U_{\et}}$ is a $G$-torsor on $U_{\et}$, so its pushforward under $g$ is finite locally constant by \cite[Exp.~XIII; Prop.~1.14, Thm.~2.4]{sga1}; this uses the additional assumption on the order of $G$ if $Y$ is non-empty.
  Hence $g_{\text{big},*} \sh I$ is representable by a finite \'etale $S$-scheme.
\end{proof}

In addition, let $\Gamma$ be a finite group acting on $G$, and on $X$ over $S$, such that $Y$ (and therefore $U$) is stable under $\Gamma$, and let $k \colon Z \to U$ be a closed immersion stable under $\Gamma$, and let $l \colon V \to U$ be its open complement.

Let $\cat T'$ denote the fppf stack of $\Gamma$-equivariant $l_!G$-torsors on $U_{\et}$; i.e.~its objects are pairs $(T,\sh F)$ of an $S$-scheme $T$ and a $\Gamma$-equivariant $l_!A$-torsor $\sh F$ on $(U \times_S T)_{\et}$, and the morphisms $(T,\sh F) \to (T', \sh F')$ are the pairs of a morphism $\phi \colon T \to T'$ and a $\Gamma$-equivariant isomorphism $\phi^{-1} \sh F' \to \sh F$.
We show that $\cat T'$ too has a representable and finite \'etale diagonal.

Again, without loss of generality, we will only consider the relevant Isom-sheaves on $S$.
Let $\sh F$ and $\sh F'$ be $\Gamma$-equivariant $l_!\sh G$-torsors on $U_{\et}$, and let $\sh I'$ denote the sheaf on $\site U_{\fppf}$ sending $\phi \colon T \to U$ to the set of $\Gamma$-invariant isomorphisms $\phi^{-1} \sh F \to \phi^{-1} \sh F'$ of $l_!G$-torsors.

\begin{cor}
  The sheaf $g_{\bsit,*} \sh I'$ on $\site S_{\fppf}$ is representable by a finite \'etale $S$-scheme.
\end{cor}

\begin{proof}
  By the last step in the proof of \autoref{d:galois-general}, we see that we can write $g_{\bsit,*} \sh I'$ as a finite limit of finite \'etale $S$-schemes, which therefore is finite \'etale over $S$ as well.
\end{proof}

Finally, let $k$ be a field, let $f \colon X \to \Spec k$ be a smooth connected curve, let $i \colon Z \to X$ be a closed immersion, and let $j \colon U \to X$ denote its open complement.
Let $\sh G$ be a finite locally constant sheaf of groups on $U_{\et}$, let $g \colon V \to U$ be a finite \'etale Galois cover with finite Galois group $\Gamma$ such that $g^{-1} \sh G$ is constant, say with $\Gamma$-module of connected components $G$, and let $Y$ be the normal completion of $V$; by replacing $k$ by a finite purely inseparable extension if necessary, we may assume $Y$ is smooth over $k$.
Let $i' \colon W \to Y$ be the base change of $i$ to $Y$, let $j' \colon V \to Y$ be the canonical open immersion, and note that $Z$ and $W$ are automatically topologically finite \'etale over $k$.

Now in \autoref{s:galois-general}, under the additional assumption that $X$ and $Y$ are curves that have smooth normal completions, we see that the stack of $j_!\sh G$-torsors is equivalent to that of $\Gamma$-equivariant $j'_! G$-torsors by \autoref{d:galois-general}, which satisfies the condition of \autoref{d:stack-key-lemma} by the above.
We will take this as the starting point of our computation in the next section.

\section{Computation of a groupoid scheme} \label{h:groupoid}

We will consider the following situation.

\begin{situ} \label{s:groupoid}
  Let $k$ be a field.
  Suppose we have a finite group $\Gamma$, and a diagram of connected curves over $k$
  \[
    \begin{tikzcd}
      V \ar[r] \ar[d,"g"'] & Y \ar[r] \ar[d,"f"] & \nc{Y} \ar[d] \\
      U \ar[r] \ar[d] & X \ar[r] \ar[d] & \nc{X} \ar[d] \\
      \bb P^1_k - S_0 - S_\infty \ar[r] & \bb P^1_k - S_\infty \ar[r] & \bb P^1_k
    \end{tikzcd}
  \]
  in which:
  \begin{itemize}
    \item $S_0 \in \set*{0,\emptyset}$; $S_\infty \in \set*{\infty,\emptyset}$;
    \item all squares are cartesian;
    \item $\nc{X} \to \bb P^1_k$ is finite locally free;
    \item $V \to U$ is finite \'etale Galois with Galois group $\Gamma$;
    \item $\nc{Y}$ is the normal completion of $V$.
  \end{itemize}
  Assume that $\nc{X}$ and $\nc{Y}$ are smooth over $k$.

  Let $\sh G$ be a finite locally constant sheaf of groups on $U_{\et}$ such that $g^{-1} \sh G$ is constant, say with group of connected components $G$.
\end{situ}

We assume the situation above is given by the smooth curves $\nc{X}$ and $\nc{Y}$ given using the description in \autoref{h:standard-modules}, the finite group $\Gamma$ and its action on $\nc{Y}/\nc{X}$, the group $G$ with $\Gamma$-action, and the choice of sets $S_0$ and $S_1$;  note that we haven't explained yet how to decide whether such an input is valid, but the characterisations in this section and the introduction of the next section will allow us to do so.
We wish to compute the Galois set $\hl^1(X,j_!\sh G)$, keeping \autoref{d:stack-key-lemma} in mind (after base change to $k^{\perf}$).
We therefore would like to give a description of, for every perfect field extension $l$ of $k$, the category $\cat T(l)$ of $j_!\sh G$-torsors on $(X_l)_{\et}$, purely in terms of (commutativity relations between) morphisms of vector bundles over $\bb P^1_l$ and free modules over $l$, and we can use \autoref{h:standard-modules} to construct a groupoid with the desired properties.

By \autoref{d:galois-general}, $\cat T(l)$ is equivalent to that of $\Gamma$-equivariant $G$-torsors on $(Y_l)_{\et}$, together with a $\Gamma$-equivariant section from $\nc{Y_l} \times_{\bb P^1_l} S_0$.
(We note that in case $S_0 = \emptyset$, the empty morphism is $\Gamma$-equivariant.)
By taking normal completions, we obtain the following.

\begin{prop}
  Let $l$ be a perfect field extension of $k$.
  Then the category $\cat T(l)$ is equivalent to that of finite locally free $\bb P^1_l$-schemes $T$, smooth over $l$, together with a $\Gamma$-equivariant $G$-action, a $\Gamma$-equivariant morphism $T \to \nc{Y_l}$ and a $\Gamma$-equivariant section $\nc{Y_l} \times_{\bb P^1_l} S_0 \to T \times_{\bb P^1_l} S_0$, such that $T \times_{\bb P^1_l} (\bb P^1_l - S_\infty)$ is a $G$-torsor on $(Y_l)_{\et}$; here, the morphisms are the $\Gamma$-equivariant, $G$-equivariant morphisms of $\nc{Y_l}$-schemes.
\end{prop}

So aside from the conditions ``smooth over $l$'' and ``$T \times_{\bb P^1_l} (\bb P^1_l - S_\infty)$ is a $G$-torsor on $(Y_l)_{\et}$'', the data in the description above can easily be expressed in terms of morphisms of vector bundles on $\bb P^1_l$, and the relations in the description above can be easily expressed in terms of commutativity relations between these morphisms.

\subsection{Torsors} \label{h:groupoid-torsor}

Let us first consider the condition ``$T \times_{\bb P^1_l} (\bb P^1_l - S_\infty)$ is a $G$-torsor on $(Y_l)_{\et}$''.
To this end, we first express the condition ``$T$ is finite locally free over $\bb P^1_l$ of constant rank'', using the fibre of $Y_l$ above $0 \in \bb P^1_l$; this is not automatic as we didn't assume $\nc{Y}$ to be geometrically connected.

\begin{lem}
  Let $S$ be a scheme, let $X$ be a finite locally free $\bb P^1_S$-scheme that is smooth over $S$.
  Let $Y$ be an $X$-scheme that is finite locally free over $\bb P^1_S$, such that $Y \times_{\bb P^1_S} 0$ is finite locally free over $X \times_{\bb P^1_S} 0$ of constant rank $n$.
  Then $Y$ is a finite locally free $X$-scheme of constant rank $n$.
\end{lem}

\begin{proof}
  We can check fibrewise on $S$ that $X \times_{\bb P^1_S} 0$ intersects all components of $X$, from which our claim follows.
\end{proof}

Since finite locally free modules over an Artinian ring are free, we have the following.

\begin{cor}
  Let $l$ be a perfect field extension of $k$.
  Then the category of finite locally free $\nc{Y_l}$-schemes of constant rank is equivalent to that of finite locally free $\bb P^1_l$-schemes $T$ together with a morphism $T \to \nc{Y_l}$ and an $\sh O(\nc{Y_l} \times_{\bb P^1_l} 0)$-basis for $\sh O(T \times_{\bb P^1_l} 0)$ (morphisms in this category are simply morphisms of $\nc{Y_l}$-schemes).
\end{cor}

Next, we want to express the condition ``$T \times_{\bb P^1_l} (\bb P^1_l - S_\infty)$ is \'etale over $Y_l$'' in terms of vector bundles on $\bb P^1_l$.
To this end, we will use the {\em transitivity of the discriminant}.

First, we recall the definitions of the {\em discriminant} and the {\em norm} of a finite locally free morphism $Y \to X$.
Recall that, for a finite locally free morphism $Y \to X$ of schemes, we view $\sh O_Y$ as a (finite locally free) $\sh O_X$-algebra.

\begin{defn}
  Let $f \colon Y \to X$ be a finite locally free morphism of schemes of constant rank, and let $\mu$ be the multiplication map $\sh O_Y \otimes_{\sh O_X} \sh O_Y \to \sh O_Y$.
  The {\em trace form} $\tau_f$ of $f$ is the morphism $\sh O_Y \to \iHom_{\sh O_X}(\sh O_Y, \sh O_X)$ corresponding to the composition $\Tr_f \mu \colon \sh O_Y \otimes_{\sh O_X} \sh O_Y \to \sh O_X$.
  The {\em discriminant} $\Delta_f$ of $f$ is the determinant (over $\sh O_X$) of the trace form $\tau_f$.
\end{defn}

\begin{defn}[{cf.~\cite{ferrand98}}]
  Let $f \colon Y \to X$ be a finite locally free morphism of schemes of constant rank, and let $\sh L$ be a line bundle on $Y$.
  The {\em norm} $\Nm_f \sh L$ of $\sh L$ is the line bundle
  \[
    \iHom_{\sh O_X}(\det_{\sh O_X} f_* \sh O_Y, \det_{\sh O_X} f_* \sh L).
  \]
\end{defn}

Let $f \colon Y \to X$ be a finite locally free morphism of schemes of constant rank, and let $\sh E$ and $\sh F$ be finite locally free $\sh O_Y$-modules of the same constant rank.
By \cite[Eq.~7.1.1]{deligne87} and the fact that norms (of line bundles) commute with tensor products and duals (see \cite[Sec.~6.5]{ega2} and \cite[Prop.~3.3]{ferrand98}), we see that there is a unique isomorphism
\[
  \iHom_{\sh O_X}(\det\nolimits_{\sh O_X} \sh E, \det\nolimits_{\sh O_X} \sh F) = \iHom_{\sh O_X}(\Nm_f \det\nolimits_{\sh O_Y} \sh E, \Nm_f \det\nolimits_{\sh O_Y} \sh F)
\]
satisfying the following properties.
\begin{itemize}
  \item It is compatible with base change by open immersions.
  \item For any isomorphism $\alpha \colon \sh F \to \sh E$, we have induced isomorphisms
    \[
      \iHom_{\sh O_X}(\det\nolimits_{\sh O_X} \sh E, \det\nolimits_{\sh O_X} \sh F) \to \iEnd_{\sh O_X}(\det\nolimits_{\sh O_X} \sh E)
    \]
    and
    \[
      \iHom_{\sh O_X}(\Nm_f \det\nolimits_{\sh O_Y} \sh E, \Nm_f \det\nolimits_{\sh O_Y} \sh F) \to \iEnd_{\sh O_X}(\Nm_f \det\nolimits_{\sh O_Y} \sh E).
    \]
    Therefore they induce isomorphisms
    \[
      \iIsom_{\sh O_X}(\det\nolimits_{\sh O_X} \sh E, \det\nolimits_{\sh O_X} \sh F) \to \iAut_{\sh O_X}(\det\nolimits_{\sh O_X} \sh E) = \bb G_{m,X}
    \]
    and
    \[
      \iIsom_{\sh O_X}(\Nm_f \det\nolimits_{\sh O_Y} \sh E, \Nm_f \det\nolimits_{\sh O_Y} \sh F) \to \iAut_{\sh O_X}(\Nm_f \det\nolimits_{\sh O_Y} \sh E) = \bb G_{m,X}.
    \]
    These isomorphisms are equal under the given identification.
\end{itemize}

Therefore, we have the following.
\begin{cor}
  Let $f \colon Y \to X$ be a finite locally free morphism of schemes of constant rank, and let $\sh E$ be a finite locally free $\sh O_Y$-module of constant rank $r$.
  Then
  \begin{align*}
    \det\nolimits_{\sh O_X} \sh E &= \Nm_f \det\nolimits_{\sh O_Y} \sh E \otimes_{\sh O_X} (\det\nolimits_{\sh O_X} \sh O_Y)^{\otimes r} \\
    \iHom_{\sh O_X}(\det\nolimits_{\sh O_X} \sh E,\sh O_X) &= \Nm_f \det\nolimits_{\sh O_Y} \iHom_{\sh O_Y}(\sh E, \sh O_Y) \otimes_{\sh O_X} \pth*[big]{\iHom_{\sh O_X}(\det\nolimits_{\sh O_X} \sh O_Y, \sh O_X)}^{\otimes r}
  \end{align*}
\end{cor}

Using the two identifications above, we may now state the transitivity of the discriminant.
A proof can be found in e.g.~\cite[Sec.~4.1]{lieblich00}.

\begin{thm}[Transitivity of the discriminant]
  Let $f \colon Y \to X$ and $g \colon Z \to Y$ be finite locally free morphisms of schemes of constant rank, and suppose that $g$ has rank $r$.
  Then
  \[
    \Delta_{fg} = \Nm_f \Delta_g \otimes \Delta_f^{\otimes r}.
  \]
\end{thm}

\begin{cor} \label{d:etale-discriminant}
  Let $f \colon Y \to X$ and $g \colon Z \to Y$ be finite locally free morphisms of schemes of constant rank, and suppose that $g$ has rank $r$.
  Then $g$ is \'etale if and only if we have $\det_{\sh O_X} \sh O_Z \iso (\det_{\sh O_X} \sh O_Y)^{\otimes r}$ and $\Delta_{fg}$ and $\Delta_f^{\otimes r}$ differ by a unit.
\end{cor}

Therefore we have the following.

\begin{prop}
  Let $l$ be a perfect extension of $k$.
  Then the category of finite \'etale $\nc{Y_l}$-schemes is equivalent to the full subcategory of that of finite locally free $\nc{Y_l}$-schemes $T$ of constant rank (say $r$) such that $\det_{\sh O_{\bb P^1_l}} \sh O_T \iso (\det_{\sh O_{\bb P^1_l}} \sh O_{\nc{Y_l}})^{\otimes r}$ and $\Delta_{T/\bb P^1_l}$ and $\Delta_{\nc{Y_l}/\bb P^1_l}^{\otimes r}$ differ by a unit.
\end{prop}

Note that the condition on the determinants is simply a condition on the types of the standard modules over $l$ isomorphic to $\sh O_T$ and $\sh O_{\nc{Y_l}}$, so if $S_\infty = \emptyset$, this gives an expression of the desired form.
If $S_\infty = \infty$, then we use the following instead.

\begin{prop}
  Let $l$ be a perfect extension of $k$, and assume that $S_\infty = \infty$.
  Then the category of finite locally free $\nc{Y_l}$-schemes \'etale over $Y_l$ is equivalent to the full subcategory of that of finite locally free $\nc{Y_l}$-schemes $T$ of constant rank (say $r$) such that $\Delta_{T/\bb P^1_l}$ and $\Delta_{\nc{Y_l}/\bb P^1_l}^{\otimes r}$ differ by a unit times a power of $y$.
\end{prop}

\begin{proof}
  It suffices to show that for integers $a,b$, a map $\phi \colon \sh O_{\bb P^1_l}(b) \to \sh O_{\bb P^1_k}(a)$ is an isomorphism when restricted to $\bb A^1_l$ if and only if it is given by multiplication by $s y^{a-b}$ with $s \in l^\times$.
  Since $y$ becomes invertible after restricting to $\bb A^1_l$, it follows that if $\phi$ is multiplication by $s y^{a-b}$, then $\phi|_{\bb A^1_l}$ is an isomorphism.
  Conversely, $\phi$ is multiplication by some $f \in l[x,y]_{a-b}$, which after restriction becomes the multiplication by $f(x,1)$ map $l[x] \to l[x]$.
  Since this map is an isomorphism, $f(x,1)$ must be an invertible constant in $l[x]$, i.e.~$f = s y^{a-b}$ for some $s \in l^{\times}$.
\end{proof}

We are almost ready to express the condition ``$T \times_{\bb P^1_l} (\bb P^1_l - S_\infty)$ is a $G$-torsor on $(Y_l)_{\et}$'' in terms of vector bundles on $\bb P^1_l$.

\begin{lem} \label{d:torsor-fibres}
  Let $f \colon Y \to X$ be a morphism of schemes, and let $G$ be a finite group acting on $Y/X$.
  Then $Y$ is a $G$-torsor on $X$ if and only if $f$ is flat, surjective, locally of finite presentation, and $G$ acts freely and transitively on geometric fibres.
\end{lem}

\begin{proof}
  The necessity of the condition is clear.
  Hence suppose that $f$ is flat, surjective, locally of finite presentation, and $G$ acts freely and transitively on geometric fibres.
  Then for any geometric point $\gpt x$ of $S$, $Y_{\gpt x}$ is the trivial $G$-torsor, hence \'etale.
  As the property of being \'etale is fpqc local on the base, it follows that all fibres of $f$ are \'etale, and since $f$ is flat and locally of finite presentation, it follows that $f$ is finite \'etale.

  Now consider the morphism $\phi \colon G \times Y \to Y \times_X Y$ of finite \'etale $Y$-schemes given on the functor of points by $(g,y) \mapsto (gy,y)$, where the occurring schemes are viewed as $Y$-schemes via the projection on the second coordinate.
  Then $\phi$ is itself finite \'etale surjective, and as $G \times Y$ and $Y \times_X Y$ have the same rank over $Y$, it follows that $\phi$ is an isomorphism.
  After base change with itself, it admits a section, so as $f$ is finite \'etale, it also follows that $Y$ is a $G$-torsor, as desired.
\end{proof}

\begin{lem} \label{d:torsor-locus}
  Let $f \colon Y \to X$ be a finite \'etale morphism of schemes of constant rank $n$, and let $G$ be a finite group of order $n$ acting on $Y/X$.
  Then the locus in $X$ where $f$ is a $G$-torsor is open and closed in $X$.
\end{lem}

\begin{proof}
  Consider the locus $U$ in $Y \times_X Y$ on which the morphism $G \times Y \to Y \times_X Y$ given on the functor of points by $(g,y) \mapsto (gy,y)$ is an isomorphism (i.e.~where the rank is equal to 1).
  It is an open and closed subset of $Y \times_X Y$ as this morphism is finite \'etale.
  As the rank of $f$ is equal to $n$, the $X$-locus where the same morphism is an isomorphism is the image of $U$ in $X$, and hence is open and closed as well.
  This locus equals the $X$-locus where $f$ is a $G$-torsor, as desired.
\end{proof}

Therefore we have the following.

\begin{cor}
  Let $l$ be a perfect extension of $k$.
  Then the category of finite locally free $G$-equivariant $\nc{Y_l}$-schemes $T$ such that $T \times_{\bb P^1_l} (\bb P^1_l - S_\infty)$ is a $G$-torsor on $(Y_l)_{\et}$ is equivalent to the category of finite locally free $G$-equivariant $\nc{Y_l}$-scheme $T$ such that $T \times_{\bb P^1_l} (\bb P^1_l - S_\infty)$ is \'etale, and such that $T \times_{\bb P^1_l} 0$ is a $G$-torsor on $(\nc{Y_l} \times_{\bb P^1_l} 0)_{\et}$.
\end{cor}

Since in the description of finite locally free $\nc{Y_l}$-schemes $T$, an $\sh O(\nc{Y_l} \times_{\bb P^1_l} 0)$-basis for $\sh O(T \times_{\bb P^1_l} 0)$ occurred, in terms of which we can express the condition that $T \times_{\bb P^1_l} 0$ is a $G$-torsor on $(\nc{Y_l} \times_{\bb P^1_l} 0)_{\et}$.

\subsection{Smoothness at $\infty$} \label{h:groupoid-smooth}

Finally, we consider the condition ``$T$ is smooth over $l$''.
If $S_\infty = \emptyset$, then this follows automatically from $T$ having to be \'etale over $\nc{Y_l}$, so assume that $S_{\infty} = \infty$.
As $T \times_{\bb P^1_l} \bb A^1_l$ has to be \'etale over $Y_l$, it suffices to consider the condition ``$T$ is smooth over $l$ at $T \times_{\bb P^1_l} \infty$''.

To this end, assume that we have a scheme $S$, a positive integer $r$, and the structure of an algebra $\sh A$ on $\sh O_S^r$, given by, for the standard basis $e_1,\dotsc,e_r$ on $\sh O_S^r$, $e_ie_{i'} = \sum_j \mu_{jii'} e_j$ and $1 = \sum_j \epsilon_j e_j$.
Then the relative differentials $\Omega_{\sh A/\sh O_S}$ over $S$ are generated by the $\dd\!e_j$, with relations $e_{i'}\dd\!e_i + e_i\dd\!e_{i'} - \sum_j \mu_{jii'} \dd\!e_j = 0$ for all $i,i'$ and $\sum_j \epsilon_j \dd\!e_j = 0$.
Therefore we get a canonical presentation $\omega_{\sh A/\sh O_S} \colon \sh A^{r^2+1} \to \sh A^r$ of the $\sh A$-module $\Omega_{\sh A/\sh O_S}$, which is compatible with base change.

\begin{prop}
  Let $l$ be any extension of $k$.
  The category of finite locally free $\bb P^1_l$-schemes $T$ smooth over $l$ at $T \times_{\bb P^1_l} \infty$ is equivalent to that of finite locally free $\bb P^1_l$-schemes $T$, together with morphisms
  \[
    i \colon \sh O_{T \times_{\bb P^1_l} \infty^{(2)}} \to O_{T \times_{\bb P^1_l} \infty^{(2)}}^{2r}, \qquad j \colon O_{T \times_{\bb P^1_l} \infty^{(2)}}^{2r} \to O_{T \times_{\bb P^1_l} \infty^{(2)}}^{(2r)^2+2}
  \]
  of $O_{T \times_{\bb P^1_l} \infty^{(2)}}$-modules such that $(\omega_{O_{T \times_{\bb P^1_l} \infty^{(2)}}/l} \oplus i)j$ is the identity on $O_{T \times_{\bb P^1_l} \infty^{(2)}}^{2r}$;
  the morphisms in the latter category are simply the morphisms of $\bb P^1_l$-schemes.
\end{prop}

\begin{proof}
  We will first show that $T$ is smooth over $l$ at $T \times_{\bb P^1_l} \infty$ if and only there exist $i$ and $j$ as in the proposition.

  Write $B$ for the ring of global sections of $T \times_{\bb P^1_l} (\bb P^1_l - 0)$, and note that it is a finite locally free $l[y]$-algebra.
  Then $T \times_{\bb P^1_l} \infty^{(2)} = \Spec B/y^2B$.
  First suppose that there exist morphisms 
  \[
    i \colon (B/y^2B) \to (B/y^2B)^{2r}, \qquad j \colon (B/y^2B)^{2r} \to (B/y^2B)^{(2r)^2+2}
  \]
  such that for the canonical presentation $\omega_{(B/y^2B)/l} \colon (B/y^2B)^{(2r)^2+1} \to (B/y^2B)^{2r}$ of $\Omega_{(B/y^2B)/l}$ as a $(B/y^2B)$-module, we have $(\omega_{(B/y^2B)/l} \oplus i) j = \id$.
  It immediately follows that $\Omega_{(B/y^2B)/l}$ is generated by one element as $B/y^2B$-module.

  Conversely, if $\Omega_{(B/y^2B)/l}$ is generated by one element, we let $i$ be a morphism from $(B/y^2B)$ to $(B/y^2B)^{2r}$ sending $1$ to (a lift of) a generator of $\Omega_{(B/y^2B)/l}$.
  Hence $(\omega_{(B/y^2B)/l} \oplus i)$ is a surjective morphism to a free $B/y^2B$-module, so it has a section $j$, as desired.

  It remains to show that $\Omega_{(B/y^2B)/l}$ is generated as a $B/y^2B$-module by one element if and only if $T$ is smooth over $l$ at all points lying over $\infty \in \bb P^1_l$.
  Note that 
  we have an isomorphism
  \[
    \Omega_{B/l} \otimes_{B} (B/yB) \to \Omega_{(B/y^2B)/l} \otimes_{B/y^2B} (B/yB),
  \]
  and that by Nakayama's lemma, the right hand side (and therefore the left hand side) is generated as a $B/yB$-module by one element if and only if $\Omega_{(B/y^2B)/l}$ is generated as a $B/y^2B$-module by one element.
  Therefore, again by Nakayama's lemma, there exists some $f \in 1+yB$ such that $\Omega_{B/l} \otimes_B B_f$ is generated as a $B_f$-module by one element.
  So the left hand side is a $B/yB$-module generated by one element if and only if there exists a neighbourhood of $T \times_{\bb P^1_l} \infty$ that is smooth over $l$, which holds if and only if $T$ is smooth over $l$ at all points lying over $\infty \in \bb P^1_l$.

  So now we have a forgetful functor from the category of finite locally free $\bb P^1_l$-schemes $T$ together with morphisms $i$ and $j$ as in the proposition, to that of finite locally free $\bb P^1_l$-schemes $T$ smooth over $l$ at $T \times_{\bb P^1_l} \infty$, which is essentially surjective by the above, and fully faithful by construction.
\end{proof}

\subsection{Bounds on types}

In order to construct a groupoid scheme with the desired properties using the above, we first need to bound the number of possible types.

\begin{lem} \label{d:types-nonpositive}
  Let $S$ be a scheme, let $a$ be a finite sequence of integers, let $X$ be a finite locally free $\bb P^1_S$-scheme of which the underlying $\sh O_{\bb P^1_S}$-modules is standard of type $a$, and suppose that $X$ has geometrically reduced fibres over $S$.
  Then $a$ is non-positive (i.e.~all of its elements are non-positive). 
\end{lem}

\begin{proof}
  By taking a geometric fibre if necessary, we assume without loss of generality that $S$ is the spectrum of an algebraically closed field $k$.
  Let $X_1, \dotsc, X_t$ be the connected components of $X$.
  Then there exist finite sequences $a_1,\dotsc,a_t$ such that for all $i$, the algebra $\sh O_{X_i}$ is of type $a_i$.
  These have the property that their concatenation is equal to $a$ up to a permutation.
  Hence we assume without loss of generality that $X$ is connected.
  In this case $X$ is a reduced curve over $S$, so $\sh O_X(\bb P^1_S) = \sh O_X(X) = \sh O_S(S) = k$, where $\pi$ is the structure morphism of $X$, so we deduce that $a$ is non-positive.
\end{proof}

\begin{rem}
  Of course, the converse is not true; a counterexample is the \mbox{$\sh O_{\bb P^1_k}$-module} $\sh O_{\bb P^1_k} \oplus \sh O_{\bb P^1_k}(-1) \epsilon$ with multiplication given by $\epsilon^2 = 0$.
\end{rem}

In \autoref{s:groupoid}, let $l$ be a perfect extension of $k$, and let $T$ be the normal completion of a $j'_!G$-torsor on $(Y_l)_{\et}$.
Let $a$ be the type of $\sh O_Y$, let $b$ be the type of $\sh O_T$, and let $s,t$ be their respective lengths.
Then by the above, both $a$ and $b$ are non-negative.
As the degree of the finite locally free morphism $T \to Y_l$ is equal to $\#G$, we see that $t = s \cdot \#G$.
Moreover, if $S_\infty = \emptyset$, then by \autoref{d:etale-discriminant}, we have $\sum_j b_j = \#G \cdot \sum_i a_i$; so up to permutation, we only have finitely many possibilities for $b$.
So suppose that $S_\infty = \infty$.

\begin{lem} \label{d:types-genus}
  Let $S$ be a scheme, let $a$ be a finite sequence of integers, and let $X$ be finite locally free $\bb P^1_S$-scheme such that $\sh O_X$ is a standard module over $S$ of type $a$, where $a$ has length $s$, and such that $X$ is smooth over $S$.
  Then $X$ is a family of curves over $S$ of Euler characteristic $s + \sum_i a_i$.
\end{lem}

\begin{proof}
  It suffices to check this on geometric fibres, so we may assume that $S$ is the spectrum of an algebraically closed field $k$.
  Then
  \begin{align*}
    \dim_k \hl^0(X,\sh O_X) - \dim_k \hl^1(X,\sh O_X) &= \dim_k \hl^0\pth*[big]{\bb P^1_k,\sh O_{\bb P^1_k}(a)} - \dim_k \hl^1\pth*[big]{\bb P^1_k,\sh O_{\bb P^1_k}(a)} \\
        &= \sum_{i} (1+a_{i}) \\
        &= s + \sum_i a_i.
  \end{align*}
\end{proof}

\begin{prop} \label{d:types-bound}
  Let $S$ be a scheme, and let $Y \to X$ be a morphism of finite locally free $\bb P^1_S$-schemes, with $\sh O_X$ and $\sh O_Y$ standard modules over $S$ of respective types $a$ and $b$, which have respective lengths $s,t$.
  Let $G$ be a finite group of order invertible in $S$ acting on $Y$ over $X$, such that $Y \times_{\bb P^1_S} \bb A^1_S$ is a $G$-torsor over $X \times_{\bb P^1_S} \bb A^1_S$.
  Then
  \[
    \sum_j b_j \geq \#G \sum_i a_i - \tfrac12 t.
  \]
\end{prop}

\begin{proof}
  It suffices to check this on geometric fibres, so we may assume that $S$ is the spectrum of an algebraically closed field $k$.
  As $G$ acts transitively on $Y$ over $X$, and the order of $G$ is invertible in $k$, it follows that $Y$ is tamely ramified over $X$.
  Therefore the ramification degree of $Y$ over $X$ is at most $t$, as $Y \times_{\bb P^1_k} \bb A^1_k$ is \'etale over $X \times_{\bb P^1_k} \bb A^1_k$, and $Y$ has degree $t$ over $\bb P^1_k$.
  So by the Riemann-Hurwitz formula, we have
  \[
    -2t - 2\sum_{j} b_j \leq -2\tfrac ts s - 2\tfrac ts \sum_i a_i + t,
  \]
  as desired (note that $t = s \cdot \#G$).
\end{proof}

So therefore, also in the case that $S_\infty = \infty$, we see that there are only finitely many possibilities for the type $b$ of $T$.

\subsection{Computation of the groupoid scheme}

Now we see that, in \autoref{s:groupoid}, the description of the category of $j_!\sh A$-torsors on $X_{\et}$ in terms of vector bundles on $\bb P^1_k$ gives, for each of the (finitely many) possibilities for the type $b$, a groupoid scheme $\mdl R_b \rightrightarrows \mdl U_b$ of which $\mdl R_b$ and $\mdl U_b$ are (explicitly given) closed subschemes of some $\bb A^N_k$.
Let $\mdl R = \coprod_b \mdl R_b$ and $\mdl U = \coprod_b \mdl U_b$.

\begin{prop}
  The groupoid scheme $\mdl R \rightrightarrows \mdl U$ satisfies the conditions of \autoref{d:stack-key-lemma}.
  Moreover, for fixed $b$, every isomorphism class in $\mdl U_{b,k^\alg}$ has the same dimension.
\end{prop}

\begin{proof}
  By construction, it remains to check that the isomorphism classes in $\mdl U_b$ are irreducible for all finite sequences $b$ of integers.
  So let $a$ be the type of the underlying $\bb P^1_k$-vector bundle of $\nc{Y}$, and $s$ its length, and let $b$ be a finite sequence of integers, and $t$ its length.
  Denote the morphisms $\mdl R_b \rightrightarrows \mdl U_b$ by $\alpha_b, \omega_b$, with $\alpha_b$ sending a morphism to its source, and with $\omega_b$ sending a morphism to its target.
  For this, it suffices to show that for all $x \in \mdl U_b(k^{\alg})$, the geometric fibre $H$ of $\alpha_b$ above $x$ is irreducible, since the image $\omega_b(H)$ in $\mdl U_{b,k^\alg}$ is by definition the isomorphism class of $x$.

  Note that, for any $k^\alg$-scheme $S$, giving an isomorphism with fixed source (say with underlying $\nc{Y}_S$-scheme $T$) in the groupoid $\mdl R_b(S) \rightrightarrows \mdl U_b(S)$ is, by transport of structure, the same as giving:
  \begin{itemize}
    \item an $\sh O_{\bb P^1_S}$-linear automorphism of $\sh O_{\bb P^1_S}(b)$;
    \item an $\sh O_{\nc{Y}_S \times_{\bb P^1_S} 0}$-linear automorphism of $\sh O_{\nc{Y}_S \times_{\bb P^1_S} 0}^{\#A}$;
    \item if $S_\infty = \infty$, morphisms 
      \[
        i \colon \sh O_{T \times_{\bb P^1_S} \infty^{(2)}} \to O_{T \times_{\bb P^1_S} \infty^{(2)}}^{2\#A}, \qquad j \colon O_{T \times_{\bb P^1_S} \infty^{(2)}}^{2\#A} \to O_{T \times_{\bb P^1_S} \infty^{(2)}}^{(2\#A)^2+2}
      \]
      such that $(\omega_{T \times_{\bb P^1_S} \infty^{(2)}/S} \oplus i)j = \id$.
  \end{itemize}

  In case $S_\infty = \emptyset$, we obtain an obvious isomorphism $H \to H_1 \times H_2$, and in case $S_\infty = \infty$, we obtain an obvious map $H \to H_1 \times H_2 \times H_3$, where
  \begin{itemize}
    \item $H_1$ is the functor sending a $k^\alg$-scheme $S$ to $\Aut_{\sh O_{\bb P^1_S}}\pth*[big]{\sh O_{\bb P^1_S}(b)}$;
    \item $H_2$ is the functor sending a $k^\alg$-scheme $S$ to $\Aut_{\sh O_{\nc{Y}_S \times_{\bb P^1_S} 0}}\pth*[big]{\sh O_{\nc{Y}_S \times_{\bb P^1_S} 0}^{\#G}}$;
    \item $H_3$ is the functor sending a $k^\alg$-scheme $S$ to the subset of
      \[
        \Hom_{\sh O_{T \times_{\bb P^1_S} \infty^{(2)}}}(\sh O_{T \times_{\bb P^1_S} \infty^{(2)}},\sh O_{T \times_{\bb P^1_S} \infty^{(2)}}^{2\#G})
      \]
      of $i$ such that $\omega_{T \times_{\bb P^1_S} \infty^{(2)}} \oplus i$ is surjective.
  \end{itemize}

  First note that using the description of standard modules, we easily see that $H_1$ is representable by a finite product of factors of the form $\bb G_{m,k}$ and $\bb A^1_k$, so therefore by a smooth, irreducible $k^\alg$-scheme.
  Moreover, note that $H_2$ is isomorphic to the functor sending a $k^\alg$-scheme $S$ to $\Hom_S(\nc{Y}_S \times_{\bb P^1_S} 0, \GL_{\#G,S})$, which, as $\sh O_{\nc{Y}_S \times_{\bb P^1_S} 0}$ is finite free over $\sh O_S$ with a given basis functorial in $S$, is representable by a non-empty open subscheme of $\bb A^{s(\#G)^2}_{k^\alg}$.
  Hence $H_2$ is a smooth, irreducible $k^\alg$-scheme as well.
  Similarly, we see that $H_3$ is representable by a non-empty open subscheme of $\bb A^{4t\#G}_{k^\alg}$, and therefore by a smooth, irreducible $k^\alg$-scheme.

  Finally, we show that $H$ is a smooth, irreducible $k^\alg$-scheme.
  We do this by showing that $H$ is Zariski locally on $H_1 \times H_2 \times H_3$ isomorphic to $H_1 \times H_2 \times H_3 \times \bb A^N_{k^\alg}$ for some fixed $N$.

  First note that we have a morphism $H_3 \to \bb A^M_{k^\alg}$ of $k^\alg$-schemes, which is given on the functor of points by sending $i \in H_3(S)$ to the corresponding $4t\#G \times \pth*[big]{2t(2\#G)^2+4t}$-matrix with coefficients in $\sh O(S)$, with respect to the basis subordinate to both the standard bases and the given $k$-basis of $T \times_{\bb P^1_k} \infty^{(2)}$, so that $M = 4t^2\pth*[big]{(2\#G)^3+4\#G}$.
  So let $i \in H_3(k^\alg)$, and view it as a $4t\#G \times \pth*[big]{2t(2\#G)^2+4t}$-matrix over $k^\alg$.
  As this matrix corresponds to a surjective map of $k^\alg$-vector spaces, there is a $4t\#G \times 4t\#G$-minor which is invertible.
  Let $U \subs \bb A^M_{k^\alg}$ be the locus on which this minor is invertible, and let $V$ be the inverse image of $U$ in $H_3$; $V$ is an open neighbourhood of $i$.

  Now let $j \in H_3(V)$ be the open inclusion. 
  By construction, the kernel of $\omega_{T \times_{\bb P^1_V} \infty^{(2)}} \oplus j$ is free over $\sh O_V$.
  Since an $\sh O_{T \times_{\bb P^1_V} \infty^{(2)}}$-linear section of this map is well-defined up to a unique tuple of elements from this kernel, it follows that the inverse image of $H_1 \times H_2 \times V$ in $H$ is isomorphic to $H_1 \times H_2 \times V \times \bb A^N_{k^{\alg}}$ for some fixed $N$ that is independent of the choices made.
  Hence $H$ is a smooth, irreducible $k^\alg$-scheme, as desired.

  Finally note that the dimension of $H$ only depends on the type $b$, and that the induced morphism $H \to \mdl U_{b,k^{\alg}}$ has finite fibres, so every isomorphism class in $\mdl U_{b,k^{\alg}}$ has the same dimension.
\end{proof}

\begin{cor}
  There is a canonical bijection from $\pi_0(\mdl U_{k^\sep})$ to the set of isomorphism classes of $\sh G$-torsors on $X_{k^\sep}$.
  Moreover, all $\mdl U_b$ are equidimensional.
\end{cor}

Next, note that we now have an obvious algorithm which, given a diagram as in \autoref{s:groupoid}, computes $\mdl R$ and $\mdl U$; the remainder of this section will be devoted to bounding the complexity of this algorithm.
We will in the following restrict ourselves to the case in which $S_0 = 0$ and $S_\infty = \infty$; the bounds we obtain in this case will also hold in the other cases.

Let $a$ be the type of $\nc{Y}$, say of length $s$, and write $\gamma = \sum_i -a_i$.
Note that by \autoref{d:types-genus}, $\gamma = s - 1 + p_a(\nc{Y})$, where $p_a$ denote the arithmetic genus.
Also note that $\#\Gamma \leq s$, so by \autoref{d:tensor-type-bounds} below, the number of field elements needed to give \autoref{s:groupoid} is polynomial in $s$, $\gamma$, and $\#G$.

First, let us bound the number of possible types $b$ that can occur as the type of an object of $\mdl U$.

\begin{prop}
  The logarithm of the number of $b$ that can occur as the type of an object of $\mdl U$ is $O\pth*[big]{s\#G \log(s\gamma\#G)}$.
\end{prop}

\begin{proof}
  For convenience, write $N = \ceil*{\#G\pth*{\tfrac12s + \gamma}}$.

  By \autoref{d:types-nonpositive}, a possible type $b$ must be non-positive.
  By \autoref{d:types-bound}, a possible type $b$ must satisfy $\sum_j -b_j \leq \#G\pth*[big]{\frac12s + \gamma}$.
  Such a type corresponds to a unique tuple $(c_0,\dotsc,c_N)$ of non-negative integers with $\sum_{k=0}^N c_k = t$ and $\sum_{k=0}^N kc_k \leq N$ by setting $c_k$ to be the number of $-b_j$ equal to $k$.
  The number of tuples satisfying the first of these conditions is $\binom{N+t}{t} \leq (N+t)^t$.
\end{proof}

Next, we will bound the size of $\mdl R_b$, i.e.~for the given closed immersion $\mdl R_b \to \bb A^N_k$, the number $N$, the number of polynomials generating the defining ideal, and the degree of these polynomials.
Note that bounds for $\mdl R_b$ will also hold for $\mdl U_b$.
To this end, note that we have the following trivial bound.

\begin{lem}
  Let $a$ and $b$ be finite sequences of non-positive integers, of lengths $s$ and $t$, respectively.
  Then $\dim_k \Hom_{\sh O_{\bb P^1_k}}(\sh O_{\bb P^1_k}(a),\sh O_{\bb P^1_k}(b)) \leq st + t \sum_i -a_i$.
\end{lem}

\begin{cor} \label{d:tensor-type-bounds}
  Let $a$ be a finite sequence of non-positive integers, of length $s$.
  Then
  \begin{align*}
    \dim_k \Hom_{\sh O_{\bb P^1_k}}\pth*[big]{\sh O_{\bb P^1_k},\sh O_{\bb P^1_k}(a)} &\leq s + \sum_i -a_i \\
    \dim_k \Hom_{\sh O_{\bb P^1_k}}\pth*[big]{\sh O_{\bb P^1_k}(a),\sh O_{\bb P^1_k}(a)} &\leq s^2 + s\sum_i -a_i \\
    \dim_k \Hom_{\sh O_{\bb P^1_k}}\pth*[big]{\sh O_{\bb P^1_k}(a)^{\otimes 2},\sh O_{\bb P^1_k}(a)} &\leq s^3 + 2s^2\sum_i -a_i \\
    \dim_k \Hom_{\sh O_{\bb P^1_k}}\pth*[big]{\sh O_{\bb P^1_k}(a)^{\otimes 3},\sh O_{\bb P^1_k}(a)} &\leq s^4 + 3s^3\sum_i -a_i.
  \end{align*}
\end{cor}

Therefore, working out everything, which is straightforward but tedious, gives the following.

\begin{prop} \label{d:groupoid-bounds}
  For the given closed immersion $\mdl R_b \to \bb A^N_k$, we have $N = O\pth*[big]{s^4(\#G)^4\gamma}$, its defining ideal is given by $O\pth*[big]{s^4(\#G)^4\gamma}$ polynomials, which have degree at most $s\#G$.
\end{prop}

Note that a polynomial ring in $N$ variables has $\binom{N+d}{d}$ monomials of degree at most $d$; so by the proposition above, we see that the size of the output is
\[
  \expt\pth*[Big]{O\pth*[big]{s\#G \log(s\gamma \#G)}}.
\]

Now that we have bounds for the sizes of $\mdl R$ and $\mdl U$, we now turn to the degrees of the defining polynomials of the morphisms defining the structure of a groupoid scheme.

Recall for this that points of $\mdl R_b$ are given by two objects of $\mdl U_b$, together with an $\sh O_{\bb P^1}$-linear map connecting the two objects.
So the source and target maps $\mdl R_b \to \mdl U_b$ are induced by projections between their ambient affine spaces.
Therefore the affine $k$-scheme $\mdl R_b \times_{\mdl U_b} \mdl R_b$ of finite type is given by $O\pth*[big]{s^4(\#G)^4\gamma}$ variables, $O\pth*[big]{s^4(\#G)^4\gamma}$ relations of degree at most $s\#G$.
Moreover, the composition map $\mdl R_b \times_{\mdl U_b} \mdl R_b \to \mdl R_b$ forgets the middle object and composes the two $\sh O_{\bb P^1}$-linear maps, so it is given by polynomials of degree at most $2$.

\begin{thm}
  The obvious algorithm computes the groupoid scheme $\mdl R \rightrightarrows \mdl U$ given \autoref{s:groupoid} as input, and has arithmetic complexity 
  \[
    \expt\pth*[Big]{O\pth*[big]{s\#G\log(s\gamma\#G)}}.
  \]
\end{thm}

\begin{proof}
  Simply note that every individual coefficient can be computed in arithmetic complexity bounded by a fixed polynomial in $s,\gamma,\#G$.
\end{proof}

\section{Geometric points and first cohomology} \label{h:points-torsors}

Next, we will use the groupoid scheme $\mdl R \rightrightarrows \mdl U$ to compute $\hl^1(X_{\et},j_!\sh G)$.
We will do this in a slightly more general situation, namely the following (compare with the conditions of \autoref{d:stack-key-lemma}).

\begin{situ} \label{s:points-torsors}
  Let $k$ be a field, let $\mdl R \rightrightarrows \mdl U$ be a groupoid scheme in which the morphisms $\mdl R \to \mdl U$ are smooth with geometrically irreducible fibres, and with $\mdl R, \mdl U$ affine and equidimensional, given by at most $r$ polynomials -- which are of degree at most $d$ -- in at most $N$ variables.
  Let $\stk T$ be a stack on $(\Sch/k)_{\fppf}$ with representable and finite \'etale diagonal, and let $[\mdl U/\mdl R] \to \stk T$ be a morphism such that $p^{-1}p_* [\mdl U_{k^{\perf}}/\mdl R_{k^{\perf}}] \to p^{-1}p_* \stk T_{k^{\perf}}$ is an equivalence.
  Here, for any field $k$, $p \colon (\Sch/k)_{\fppf} \to (\Spec k)_{\et}$ denotes the change-of-site morphism for which $p_*$ is the restriction.
\end{situ}

Let us, for a finite reduced $k$-algebra $A$, denote by $A^\dagger$ the separable closure of $k$ in $A$.
Moreover, if $A$ is a finite product $\prod_i k_i$ of fields, denote by $A^{\perf}$ the product $\prod_i k_i^{\perf}$.
Suppose we are in \autoref{s:points-torsors}.
Then to any morphism $x \colon \Spec l \to \mdl U$ with $l/k$ finite, we can attach an induced morphism $\Spec l^{\perf} \to \mdl U$.
This in turn induces a morphism $\Spec l^{\perf} \to p^{-1}p_*\stk T_{k^{\perf}}$, which is \'etale as both $\Spec l^{\perf}$ and $p^{-1}p_*\stk T_{k^{\perf}}$ are \'etale over $\Spec k^{\perf}$.
We hence get a morphism $\Spec l^\dagger \to p^{-1}p_*\stk T$.

We prove a couple of lemmas regarding this construction.

\begin{lem} \label{d:points-cover}
  In \autoref{s:points-torsors}, let $\set*{x_i \colon \Spec l_i \to \mdl U}$ be a family of points on $\mdl U$.
  Then the image of $\coprod_i \Spec l_i \to \mdl U$ intersects every geometric connected component of $\mdl U$ if and only if $\coprod_i \Spec l_i^\dagger \to p^{-1}p_*\stk T$ is surjective.
\end{lem}

\begin{proof}
  We note that the image of $\coprod_i \Spec l_i \to \mdl U$ intersects every geometric connected component if and only if the image of $\coprod_i \Spec l_i^{\perf} \to \mdl U$ does so.
  This is equivalent to $\coprod_i \Spec l_i^{\perf} \to p^{-1}p_*\stk T_{k^{\perf}}$ being surjective, i.e.~to $\coprod_i \Spec l_i^\dagger \to p^{-1}p_* \stk T$ being surjective.
\end{proof}

\begin{lem} \label{d:points-intersection}
  In \autoref{s:points-torsors}, let $x \colon \Spec l \to \mdl U$ and $y \colon \Spec m \to \mdl U$ be two points on $\mdl U$.
  Let $A$ be the coordinate ring of $\alpha^{-1}x \times_{\mdl R} \omega^{-1}y$.
  Then $A^\dagger$ is the coordinate ring of $\Spec l \times_{p^{-1}p_*\mdl T} \Spec m$.
\end{lem}

\begin{proof}
  Let $x' \colon \Spec l^{\perf} \to \mdl U$ and $y' \colon \Spec m^{\perf} \to \mdl U$.
  Then $A^{\perf}$ is the coordinate ring of $\alpha^{-1}x' \times_{\mdl R} \omega^{-1}y' = \Spec l^{\perf} \times_{p^{-1}p_*\stk T_{k^{\perf}}} \Spec m^{\perf}$, so $A^\dagger$ is the coordinate ring of $\Spec l^\dagger \times_{p^{-1}p_*\stk T} \Spec m^\dagger$, being the unique finite $k$-subalgebra of $A^{\perf}$ of which the base change to $k^{\perf}$ is $A^{\perf}$.
\end{proof}

So in \autoref{s:points-torsors}, by finding enough points on $\mdl U$, one can construct a presentation of the stack $p^{-1}p_* \stk T$, which then can be used to compute $\pi_0$ of this.
Let us do so explicitly below.

\begin{prop} \label{d:points-base}
  \autoref{a:points-base} takes as input \autoref{s:points-torsors} and computes a finite set $X$ of morphisms $x_i \colon \Spec l_i \to \mdl U$ with $l_i/k$ finite, such that the induced map $\coprod_i \Spec l_i \to \pi_0(p_* \stk T)$ is surjective.
  Moreover, it does so in arithmetic complexity
  \[
    \expt\pth*[Big]{O\pth*[big]{ N^2 \log(d), e N \log(d), \log(r) }}
  \]
\end{prop}

\begin{algo} \label{a:points-base}
  Compute a Noether normalisation $\nu \colon \mdl U \to \bb A^{\dim \mdl U}$ using e.g.~\cite[Sec.~1]{dickensteinetal91}.
  Note that this also works for finite fields, but only after a base change to a finite field extension;  so for finite fields, one needs to keep track of the Galois action as well.

  Then set $R = \sh O\pth*[big]{\nu^{-1}(0)}$.
  Compute a $k$-basis for $R$ using a Gr\"obner basis computation for the ideal defining $R$.
  Compute the primary decomposition of $R$ and for each local factor $S$ of $R$, compute the composition of $\sh O(\mdl U) \to R$, $R \to S$, and $S \to S^{\red}$.
\end{algo}

\begin{proof}
  First note that, as $\mdl U$ is equidimensional, every geometric connected component maps surjectively to $\bb A^{\dim \mdl U}$.
  Hence $R$ is the ring of global sections of a closed subscheme of $\mdl U$ that intersects every geometric connected component, so this procedure indeed computes a set $X$ as desired.
  It remains to prove the claims on the arithmetic complexity.

  By \cite[Sec.~1; Sec.~3]{dickensteinetal91} (for Noether normalisation and the zero-dimensional Gr\"obner basis computation, respectively), $R$ can be computed as finite $k$-algebras in arithmetic complexity
  \[
    \expt\pth*[Big]{O\pth*[big]{ N^2 \log(d), \log(r) }}.
  \]

  Let us bound the $k$-vector space dimension of the $R$.
  First note that $\mdl U \leq N$.
  Therefore $R$ is given by at most $N$ generators, and by relations that are of degree at most $d$.
  Hence 
  \[
    \dim_k R = \expt\pth*[Big]{O\pth*[big]{ N \log(d) }}.
  \]
  So by the methods of \cite[Sec.~7]{khurimakdisi04}, we see that the primary decomposition of $R$ can be computed in arithmetic complexity
  \[
    \expt\pth*[Big]{O\pth*[big]{ N \log(d) }}.
  \]
  
  Moreover, for each of the local factors, the degree of the purely inseparable extension $l/k$ to be taken doesn't exceed $(\dim_k R)^e$, as the degree of the polynomials to be factored doesn't exceed $\dim_k R$.
  It then follows by that the maps $R \to S^{\red}$ can be computed in arithmetic complexity
  \[
    \expt\pth*[Big]{O\pth*[big]{ (e+1)N\log(d) }}.
  \]
  as $\dim_k (S \otimes_k l) \leq (\dim_k R)^{e+1}$.
\end{proof}

\begin{prop} \label{d:points-isom}
  \autoref{a:points-isom} takes as input \autoref{s:points-torsors}, $x \colon \Spec l \to \mdl U$, and $y \colon \Spec m \to \mdl U$ and computes the finite $k$-scheme $\alpha^{-1} x \times_{\mdl R} \omega^{-1} y$ in arithmetic complexity
  \[
    \expt\pth*[Big]{O\pth*[big]{ N_{x,y}^2 \log(d_{x,y}), \log(r) }},
  \]
  where $N_{x,y} = \max(N,\log [l:k],\log [m:k])$ and $d_{x,y} = \max(d,[l:k],[m:k])$.
\end{prop}

\begin{algo} \label{a:points-isom}
  We first compute for $l$ and $m$ a ``small'' set of generators.
  Start by setting $X = F = \emptyset$.
  For $s$ in a $k$-basis for $l$, compute $k[X][s] \subs l$ and the minimal polynomial $f$ of $s$ over $k[X] \subs l$, and then, if $f$ is linear, do nothing, otherwise add $s$ to $X$ and $f$ to $F$.
  Write $l = k[X]/(F)$ afterwards, and repeat this for $m$.

  Now compute $x, y$ in terms of the ``small'' descriptions of $l$ and $m$ obtained above, and compute $\alpha^{-1} x \times_{\mdl R} \omega^{-1} y$.
  Finally, compute a $k$-basis for its coordinate ring via Gr\"obner bases, and the unit and multiplication table with respect to this basis.
\end{algo}

\begin{proof}
  Note that in the first step, we write $l$ (resp.~$m$) using $O(\log [l:k])$ (resp.~$O(\log [m:k])$) generators and relations, of degree at most $[l:k]$ (resp.~$[m:k]$).
  Moreover, $\mdl U$ and $\mdl R$ are given by at most $N$ generators and $r$ relations, of degree at most $d$.
  Therefore $\alpha^{-1} x \times_{\mdl R} \omega^{-1} y$ is given by $O\pth*[big]{\log [l:k], \log [m:k], N}$ generators, $O\pth*[big]{\log [l:k], \log [m:k], N, r}$ relations, of degree at most $\max\pth*[big]{[l:k],[m:k],d}$.
  Hence the arithmetic complexity follows from \cite[Sec.~3]{dickensteinetal91} in the same way as before.
\end{proof}

As a corollary, using \autoref{d:points-cover} and \autoref{d:points-intersection}, we have the following.

\begin{cor}
  There exists an algorithm that takes \autoref{s:points-torsors} as input and computes a diagram
  \[
    \begin{tikzcd}
      Y^\dagger \ar[r,shift left] \ar[r,shift right] & X^\dagger \\
      & X \ar[u] \ar[d] \\
      & \mdl U
    \end{tikzcd}
  \]
  with $X^\dagger, Y^\dagger$ finite \'etale over $k$, $Y^\dagger \rightrightarrows X^\dagger$ a presentation for $p^{-1}p_*\stk T$, $X \to X^\dagger$ a finite purely inseparable morphism between finite $k$-schemes, and $X \to \mdl U$ having image intersecting every geometric connected component of $\mdl U$, in arithmetic complexity
  \[
    \expt\pth*[Big]{O\pth*[big]{ (e+1)N^3\log(d)^3, \log(r) }}.
  \]
\end{cor}

In the corollary above, we can assume that if $U$, $V$ are two distinct connected components of $X^\dagger$, then $\alpha^{-1}U \times_{Y^\dagger} \omega^{-1}V$ is empty, since whenever we encounter distinct connected components $U$, $V$ for which $\alpha^{-1}U \times_{Y^\dagger} \omega^{-1}V$ is not empty, then we may simply omit one of $U$, $V$.
It follows that the groupoid scheme $Y^\dagger \rightrightarrows X^\dagger$ is a finite disjoint union of groupoid schemes $Y_i^\dagger \rightrightarrows X_i^\dagger$ with finite \'etale arrows in which each $X_i$ is the spectrum of a finite separable field extension of $k$.
Therefore the problem of computing $\pi_0(p^{-1}p_*\stk T)$ reduces to computing $\pi_0$ of each of these groupoid schemes.

The following lemma suggests how to compute $\pi_0$ in this case.

\begin{lem}
  Let $S$ be a connected scheme, and let $Y \rightrightarrows X$ be groupoid scheme over $S$ with $X$ and $Y$ finite \'etale $S$-schemes.
  Then the image $R$ of $Y \to X \times_S X$ is an \'etale equivalence relation on $X$, and $\pi_0([X/Y]) = X/R$.
\end{lem}

\begin{proof}
  This is trivial once we view $X$ and $Y$ as finite $\pi_1(S)$-sets.
\end{proof}

\begin{cor} \label{d:points-groupoid-iso}
  \autoref{a:points-groupoid-iso} takes a groupoid scheme $Y \rightrightarrows X$ over $k$ with $X$ and $Y$ finite \'etale $k$-schemes, and outputs $\pi_0([X/Y])$ in arithmetic complexity polynomial in the degrees of $X$ and $Y$ over $k$.
\end{cor}

\begin{algo} \label{a:points-groupoid-iso}
  Let $l$, $B$ be the respective coordinate rings of $X$, $Y$, and let $A = l \otimes_k l$.
  Let $A = \prod_i A_i$ be the primary decomposition of $A$; since $A$ is separable over $k$, all $A_i$ are fields.

  Compute the morphism $A \to B$, and compute the set $I$ of indices $i$ for which the induced map $A_i \to B$ is non-zero.
  Set $A_I = \prod_{i \in I} A_i$.
  Compute the morphisms $l \to A_I$ sending $s \in l$ to $s \otimes 1$ and $1 \otimes s$, respectively, and return their equaliser $k'$.
\end{algo}

\begin{proof}
  Note that, since $\Spec A_I$ is the image of $Y$ in $X \times_{\Spec k} X$ by construction, $\Spec k'$ is the coequaliser of the two morphisms $\Spec A_I \to X$ constructed, in other words, it is the quotient of $X$ by the \'etale equivalence relation $\Spec A_i$ on $X$, as desired.
\end{proof}

Applying the above to the groupoid scheme obtained in \autoref{h:groupoid}, we get the following.

\begin{cor}
  There exists an algorithm that takes as input \autoref{s:groupoid} and computes a finite \'etale $k$-scheme representing $\hl^1(X_{k^{\sep},\et},j_!\sh G)$ in arithmetic complexity
  \[
    \expt\pth*[Big]{O\pth*[big]{ (e+1)s^{12}(\#G)^{12}\gamma^3 \log(s\#G)^3 }}.
  \]
\end{cor}

In order to be able to compute additional structures on $\hl^1(X_{k^{\sep},\et},j_!\sh G)$, it will turn out to be useful to compute this set as a finite $\Gal(k^{\sep}/k)$-set, together with some additional structure.
The first step in this is to compute a finite Galois extension $l/k$ such that the Galois action on $\hl^1(X_{k^{\sep},\et},j_!\sh G)$ factors through $\Gal(l/k)$.
This is done in the standard recursive way.

\begin{lem} \label{d:points-minimal-galois}
  \autoref{a:points-minimal-galois} takes as input a finite separable $k$-algebra $A$, and computes the minimal Galois extension $l/k$ such that $A \otimes_k l$ is a product of copies of $l$, in arithmetic complexity polynomial in $(\dim_k A)!$.
  If $\Spec A$ is the underlying scheme of a group scheme over $k$, then the arithmetic complexity is
  \[
    \expt\pth*[Big]{O\pth*[big]{ \log (\dim_k A)^2 }}.
  \]
\end{lem}

\begin{algo} \label{a:points-minimal-galois}
  Set $A' = A$, and compute a primary decomposition $A' = \prod_i A'_i$.
  Set $l = k$.
  While the number of factors is not $\dim_k A$, choose $A'_i$ of maximal dimension, and set $A' = A' \otimes_l A'_i$, $l = A'_i$, and compute a primary decomposition $A' = \prod_{i} A'_i$.
  Return $l$.
\end{algo}

\begin{cor}
  There exists an algorithm that takes as input a finite separable $k$-algebra $A$, and computes the corresponding finite $\Gal(k^{\sep}/k)$-set in arithmetic complexity polynomial in $(\dim_k A)!$ (or
  \[
    \expt\pth*[Big]{O\pth*[big]{ \log (\dim_k A)^2 }}
  \]
  if $\Spec A$ is the underlying scheme of a group scheme over $k$).
\end{cor}

Now by base change to $l$, we get the following.

\begin{cor} \label{d:points-cohomology}
  There exists an algorithm that takes as input \autoref{s:groupoid} and computes:
  \begin{itemize}
    \item a finite Galois extension $l/k$ such that the $\Gal(k^{\sep}/k)$-action on the finite set $\hl^1(X_{k^{\sep},\et},j_!\sh G)$ factors through $\Gal(l/k)$,
    \item the finite $\Gal(l/k)$-set $\hl^1(X_{k^{\sep},\et},j_!\sh G)$,
    \item for each $h \in \hl^1(X_{k^{\sep},\et}, j_!\sh G)$, a finite extension $l_h/l$ and a morphism $\Spec l_h \to \mdl U$ representing $h$,
  \end{itemize}
  in arithmetic complexity
  \[
    \expt\pth*[Big]{O\pth*[big]{ (e+1)s^{12}(\#G)^{12}\gamma^3 \log(s\#G)^3, (e+1) \log [l:k] }}.
  \]
\end{cor}

\section{Reductions and applications} \label{h:reduction}

In the previous sections, we assumed normal proper curves to be presented as finite locally free $\bb P^1_k$-schemes as described in \autoref{h:standard-modules}.
Alternatively, normal proper connected curves can be presented using their function fields, as a finite extension of $k(x)$, and in this case, we will present morphisms between normal proper connected curves by morphisms between their function fields.

Passing from the presentation as finite locally free $\bb P^1_k$-scheme to that as a finite extension of $k(x)$ is simple:  given a finite locally free $\bb P^1_k$-scheme $X$ with type $a$ of length $s$, one can compute the finite $k(x)$-algebra $k(X)$ corresponding to it in arithmetic complexity polynomial in $s$ and $\sum_i -a_i$, simply by, in the conventions of \autoref{h:standard-modules}, substituting $y=1$ in the multiplication table and unit defining $\sh O_X$;  this computation is functorial in $X$.

Conversely, given a finite field extension $A$ of $k(x)$ of degree $d$ defined by elements of height at most $h$, one can compute $\alpha_1,\dotsc,\alpha_n \in A$ such that $A = k(x,\alpha_1,\dotsc,\alpha_n)$, and minimal polynomials for $\alpha_{i+1}$ over $k(x,\alpha_1,\dotsc,\alpha_i)$ in arithmetic complexity polynomial in $d$, $h$, using the methods of \cite{berkowitz84};  note that $n \leq \log_2 d$, and the minimal polynomials have degree at most $d$, and their coefficients have height at most $d^3h$.
By multiplying by suitable polynomials in $k[x]$, one can make each $\alpha_{i+1}$ have a minimal polynomial of which the coefficients lie in $k[x,\alpha_1,\dotsc,\alpha_i]$, in arithmetic complexity polynomial in $d$, $h$;  the minimal polynomials in this case will have $x$-degree at most $d^5h$.
Therefore we obtain a $k[x]$-order in $A$ consisting of products of the $\alpha_i$.
Similarly, we can compute a $k[x^{-1}]$-order in $A$, in arithmetic complexity polynomial in $d$ and $h$.

Then, by \cite[Sec.~2.7]{diem08} (which we can apply since we are able to compute nilradicals of finite $k$-algebras) one can compute the corresponding maximal orders over $k[x]$ and $k[x^{-1}]$ in arithmetic complexity polynomial in $d^{e+1}$ and $h$.
Moreover, they define the same $k[x,x^{-1}]$-submodule of $A$, so it follows from \cite[Lem.~11.50]{gortzwedhorn10} that one can compute a sequence $(a_i)$ of integers, and bases $(b_i)$ and $(c_i)$ of the respective maximal orders such that $b_i = x^{a_i} c_i$ for all $i$, and therefore a presentation of the normal proper connected curve as a finite locally free $\bb P^1_k$-scheme, in arithmetic complexity polynomial in $d^{e+1}$ and $h$ as well.

In fact, as the computation of nilradicals of finite $k$-algebras as described in \autoref{h:computation} proceeds by first computing the nilradical of the base change to $k^\perf$, it follows that one can compute a purely inseparable extension $l$ of $k$ and a {\em smooth} proper connected curve with function field $k(x)l$ in arithmetic complexity polynomial in $d^{e+1}$ and $h$;  we will refer to this as the construction of a {\em smooth completion}.

As for functoriality, given a morphism $K \to L$ of function fields, with $K$ given as a finite $k(x)$-algebra, and $L$ as a finite $k(y)$-algebra, one can compute a $k(x)$-basis of $L$ (and therefore a $K$-basis of $L$) by successively computing a $k(x)$-basis of $k(x,y)$ and a $k(x,y)$-basis of $L$;  with respect to this $k(x)$-basis of $L$, the computation given above is functorial.

\begin{rem}
  The above gives us an algorithm for the computation of the normalisation (over $k$ and over $k^\perf$) of a finite locally free $\bb P^1_k$-scheme of type $a$ of length $s$, in arithmetic complexity polynomial in $s^{e+1}$ and $\sum_i -a_i$;  for the type $a'$ (of length $s'$) of the resulting normal proper curve, we have $s = s'$ and $\sum_i -a'_i \leq \sum_i -a_i$.
\end{rem}

To present divisors on proper normal curves, we will mainly use the so-called {\em free ideal presentation}.
Roughly speaking, in this presentation, divisors on a proper normal curve $X$ are given as formal sums of closed points of $X$, which in turn are given by maximal ideals of $\sh O_X$.
For more details on this and other related presentations, see e.g.~\cite{hess02}, or \cite[Ch.~2]{diem08} for a more detailed exposition.
For the purposes of this paper, we simply note that we can compute images and pre-images of closed points of a morphism of proper normal curves in arithmetic complexity polynomial in the size of the input.

Now an arbitrary normal curve $X$ will be presented by the product of the function fields of its connected components, and the finite complement of $X$ in its normal completion $\nc{X}$;  as a measure for the size of an affine curve, we take the $k(x)$-degree of the corresponding $k(x)$-algebra, an upper bound $h$ for the height of the elements of $k(x)$ defining this algebra, the number of closed points in $\nc{X} - X$, and the maximum degree of these closed points over $k$.
Morphisms $Y \to X$ between normal curves will be presented by morphisms between their normal completions, such that for every closed point in the complement of $X$ in $\nc{X}$ there is a closed point of $Y$ in $\nc{Y}$ lying over it.

\subsection{Topological invariance of the small \'etale site}

In our reduction to \autoref{s:groupoid}, we will make use of finite locally free, purely inseparable morphisms between normal proper curves and the {\em topological invariance of the small \'etale site}, which states that for a universal homeomorphism $f \colon Y \to X$, the functors $f_*$ and $f^{-1}$ are quasi-inverse functors between $\Sh(X_{\et})$ and $\Sh(Y_{\et})$.
Given a finite locally free, purely inseparable morphism $f \colon Y \to X$ between normal proper connected curves, we will make this explicit for \'etale sheaves representable by \'etale separated $X$-schemes (resp.~$Y$-schemes), i.e.~by normal curves.

For $\sh F$ an \'etale sheaf representable by an \'etale and separated $X$-scheme, the pullback is simply $Y \times_X \sh F$, which clearly can be computed in arithmetic complexity polynomial in the size of the input.

\begin{prop} \label{d:topos-pushforward}
  \autoref{a:topos-pushforward} takes a finite locally free, purely inseparable morphism $f \colon Y \to X$ between normal proper connected curves, and a normal curve $\sh F$, \'etale over $Y$, and computes $f_* \sh F$, in arithmetic complexity polynomial in the size of the input.
\end{prop}

\begin{algo} \label{a:topos-pushforward}
  Write $K$, $L$, for the function fields of $X$, $Y$, respectively, let $\sh F$ be an \'etale and separated $Y$-scheme (with normal completion $\nc{\sh F}$), and let $B$ denote its corresponding $L$-algebra.
  Let $A$ be the Weil restriction of $B$ from $L$ to $K$.
  Compute the $L$-algebra isomorphism $A \otimes_K L \to B$, and therefore a morphism $A \to B$.
  Let $\nc{\sh F}'$ denote the corresponding normal proper curve, and $\nc{\sh F} \to \nc{\sh F}'$ be the corresponding morphism.
  Output $\nc{\sh F}'$ together with the image of the complement of $\sh F$ in $\nc{\sh F}$.
\end{algo}

\begin{proof}
  The output is correct since the output $\sh F'$ needs to satisfy $\sh F' \times_X Y = \sh F$, and since taking Weil restrictions sends finite separable $L$-algebras to finite separable $K$-algebras, and is left adjoint to base changing from $K$ to $L$.
\end{proof}

\subsection{Presentation of torsors}

Now note that in \autoref{s:groupoid}, we have a second presentation of a $j_!\sh G$-torsor on $X_{k^{\sep},\et}$ (the first being as a geometric point on the groupoid scheme $\mdl R \rightrightarrows \mdl U$ constructed in the previous section).
First, any $j_!\sh G$-torsor is representable by an \'etale separated $X$-scheme and therefore by a normal curve, so we can present a $j_!\sh G$-torsor $T$ by a normal curve together with the group action $j_!\sh G \times_X T \to T$.
In fact, $j_!\sh G$ is representable by the disjoint union of $X$ (acting as the zero section) and $\sh G - 1$ (which is finite \'etale over $U$).
We indicate how to pass between these presentations.

Starting with $x \in \mdl U(l)$, we let $k'$ be the separable closure of $k$ in $l$.
Note that $x$ defines a $\Gamma$-equivariant $A$-torsor $\nc{T}$ on $\nc{Y}_{l,\et}$ together with a $\Gamma$-equivariant section $\nc{Y} \times_{\bb P^1_k} S_0 \to \nc{T}$.
Using the function field presentations, one can then compute quotients under $\Gamma$ using linear algebra, which gives us a $j_!\sh G$-torsor $T$ on $X_{l,\et}$.
Therefore we can compute the corresponding $j_!\sh G$-torsor on $X_{k',\et}$ in arithmetic complexity polynomial in $s^{e+1}, (\#G)^{e+1}, \gamma^{e+1}, [l:k]^{e+1}$.

Conversely, let $T$ be a $j_!\sh G$-torsor on $X_{k',\et}$ with $k'/k$ separable.
Pull $T$ back to a $\Gamma$-equivariant $G$-torsor on $Y_{k',\et}$, together with $\Gamma$-equivariant section $Y_{k'} \times_{\bb P^1_k} S_0 \to Y_{k'}$, and compute a smooth completion.
Now some linear algebra suffices to compute the additional data (see \autoref{h:groupoid-torsor} and \autoref{h:groupoid-smooth}) required to obtain a point of $\mdl U(k^\alg)$ in arithmetic complexity polynomial in $s^{e+1}, (\#G)^{e+1}, \gamma^{e+1}, [k':k]^{e+1}$.

Therefore, using this, \autoref{d:points-cohomology} and \autoref{d:points-isom}, we have the following.

\begin{cor} \label{d:module-cohomology}
  \autoref{a:module-cohomology} takes as input \autoref{s:groupoid} (but with $\sh A = \sh G$ a sheaf of abelian groups, so $A = G$ is abelian) and computes the $\Gal(k^\sep/k)$-module $\hl^1(X_{k^\sep,\et},j_!\sh A)$ in arithmetic complexity
  \[
    \expt\pth*[Big]{O\pth*[big]{(e+1)^3s^{16}(\#A)^{16}\gamma^4 \log(s\#A)^3}}.
  \]
\end{cor}

\begin{algo} \label{a:module-cohomology}
  It remains to compute addition of classes of $j_!\sh A$-torsors on $X_{k^\sep,\et}$.
  For this, it suffices to note that if $T_1, T_2$ are $j_!\sh A$-torsors on $X_{k',\et}$ with $k'/k$ separable, then the sum of the classes of $T_1$ and $T_2$ in $\hl^1(X_{k^\sep,\et},j_!\sh A)$ is given by the quotient of $T_1 \times_{X_{k'}} T_2$ by the $j_!\sh A$-action given by $a(t_1,t_2) = (at_1,a^{-1}t_2)$;  compute this using linear algebra over $k'(x)$, and find the element in $\hl^1(X_{k^\sep,\et},j_!\sh A)$ isomorphic to it using \autoref{a:points-isom}.
\end{algo}

\begin{proof}
  It remains to find an upper bound for $\log[l:k]$.
  Note that the cardinality of the set $\hl^1(X_{k^\sep,\et},j_!\sh A)$ is
  \[
    \expt\pth*[Big]{O\pth*[big]{s^8(\#A)^8\gamma^2 \log(s\#A),es^4(\#A)^4\gamma \log(s\#A)}}
  \]
  by \autoref{d:points-base}.
  As this set is an abelian group, it follows that 
  \[
    \log[l:k] = O\pth*[big]{(e+1)^2s^{16}(\#A)^{16}\gamma^4 \log(s\#A)^2},
  \]
  from which the arithmetic complexity follows.
\end{proof}

\subsection{Reduction to \autoref{s:groupoid}}

We will now indicate how to reduce to \autoref{s:groupoid} for ``most'' smooth connected curves $X$ over $k$, non-empty open subschemes $U$ of $X$, and finite locally constant sheaves $\sh A$ on $U_{\et}$.
We would like to use explicit computation of Riemann-Roch spaces as in \cite{hess02} to compute a suitable cover of $X$ over $\bb P^1_k$, however, this requires the curve to be given as a generically \'etale finite locally free $\bb P^1_k$-scheme (or equivalently, a separating element for the function field of $X$ must be given).

\begin{prop} \label{d:reduction-separator}
  \autoref{a:reduction-separator} takes as input a normal proper connected curve $X$ over $k$ of type $a$, and computes a finite purely inseparable extension $l/k$, a finite locally free purely inseparable morphism $X' \to X$ of degree at most $[X:\bb P^1_k]^{e+2}$, and a generically \'etale finite locally free morphism $X' \to \bb P^1_l$ of degree at most $[X:\bb P^1_k]$, in arithmetic complexity polynomial in $\sum_i -a_i$ and $[X:\bb P^1_k]^{e+2}$.
\end{prop}

\begin{algo} \label{a:reduction-separator}
  By computing the separable closure $K'$ of $k(x)$ in $K$, compute the minimal $p$-power $q$ such that $f^q \in K'$ for all $f \in K$.
  Let $L = K \cdot k^{1/q}(x^{1/q})$, which is the reduction of $K \otimes_{k(x)} k^{1/q}(x^{1/q})$, and output $l = k^{1/q}$ and $x^{1/q}$, together with a $K$-basis of $L$.
\end{algo}

\begin{proof}
  We note that $q$ is bounded from above by $[K:k(x)]$, and that therefore $L$ can be computed in arithmetic complexity polynomial in $\sum_i -a_i$ and $[K:k(x)]^{e+2}$.
\end{proof}

\begin{prop} \label{d:riemann-roch}
  \autoref{a:riemann-roch} takes as input a smooth connected curve $X$ generically \'etale over $\bb P^1_k$, given by its normal completion $\nc{X}$, a finite set $\set*{Q_1,\dotsc,Q_t}$ of closed points of $\nc{X}$, and an open immersion $j \colon U \to X$, given by a finite set $\set*{P_1,\dotsc,P_s}$ of closed points of $X$, and computes a finite locally free morphism $\pi \colon X \to \bb P^1_k$ with $\pi^{-1}(\infty) = \set*{Q_1,\dotsc,Q_t}$ and $\pi^{-1}(0) \sups \set*{P_1,\dotsc,P_s}$ in arithmetic complexity polynomial in the size of the input.
\end{prop}

\begin{algo} \label{a:riemann-roch}
  Write $Z_0 = \sum_{i=1}^s P_i$, $Z_\infty = \sum_{j=1}^t Q_j$, and let $g$ be the genus of $X$.
  Let $m$ be the smallest integer such that $m(t-1) - s > 2g - 2$ and $2^m > t$.
  Using \cite{hess02}, compute a $k$-basis $B$ for $\sh O_X(-Z_0 + mZ_\infty)$, and compute the subspaces $\sh O_X(-Z_0 + mZ_\infty - mQ_j)$ for $j=1,\dotsc,t$.
  Find a linear combination $f = \sum_{b \in B} \epsilon_b b$ with $\epsilon_b \in \set*{0,1}$ for all $b \in B$ such that $f$ is not in any of the $\sh O_X(-Z_0 + mZ_\infty - mQ_j)$ for $j=1,\dotsc,t$.

  Compute a minimal polynomial for $f$ over $k(x)$, and write it as a minimal polynomial for $x$ over $k(f)$, and compute successively a $k(f)$-basis for $k(x,f)$ and a $k(f)$-basis for $k(X)$.
  Output the corresponding map $X \to \bb P^1_k$.
\end{algo}

\begin{proof}
  Note that by choice of $m$, we see that $\sh O_X(-Z_0 + mZ_\infty - mQ_j)$ has dimension $\#B - m$, so $f$ ranges over a set of $2^{\#B}$ elements, of which at most $t2^{\#B-m}$ lie in one of the given subspaces.
  By choice of $m$, we have $t2^{\#B-m} < 2^{\#B}$, so there exists such $f$ not lying in any of the given subspaces.
\end{proof}

In particular, if either $U = X$ or $X = \nc{X}$, then we can get a finite locally free morphism $\nc{X} \to \bb P^1_k$ with $U$ and $X$ inverse images of $\bb P^1_k$, $\bb P^1_k - 0$, or $\bb P^1_k - \infty$.
Therefore, using smooth completions, we now have an obvious reduction to \autoref{s:groupoid}, and therefore the following corollary, which in turn implies \autoref{d:main}.

\begin{cor}
  Let $S_0 \in \set*[big]{\emptyset,0}$ and $S_\infty \in \set*[big]{\emptyset,\infty}$.
  There is an algorithm that takes a finite locally free $\bb P^1_k - S_\infty$-scheme $X$, smooth over $k$, and a finite \'etale commutative group scheme $\sh A$ over $U = X \times_{\bb P^1_k} (\bb P^1_k - S_0 - S_\infty)$, and computes the $\Gal(k^\sep/k)$-module $\hl^1(X_{k^\sep,\et},j_!\sh A)$ in arithmetic complexity exponential in $e$, $[X:\bb P^1_k - S_\infty]$, $[\sh A:U]^{\log [\sh A:U]}$, $\gamma_X$, and $\gamma_{\sh A}$.
  Here, $\gamma_X$ (resp.~$\gamma_{\sh A}$) is $\sum_i -a_i$, where $a$ is the type of the normal completion of $X$ (resp.~$\sh A$).
\end{cor}

\begin{proof}
  We need to prove that the size of the cover $\nc{Y}$ constructed is polynomial in $[X:\bb P^1_k - S_\infty]$, $[\sh A:U]^{\log [\sh A:U]}$, $\gamma_X$, and $\gamma_{\sh A}$.
  Recall that $\nc{Y}$ is constructed by setting $X' = X$, $\sh A' = \sh A$ and then repeatedly base changing $\sh A'/X'$ to a non-trivial connected component of $\sh A'$.

  For each such base change, let $a$ be the type of $\sh A'$, let $a'$ be the type of the chosen connected component of $\sh A'$, and let $a''$ be the type of their fibre product $\sh A''$ over $X$.
  Note that as then $\sh O_{\bb P^1_k}(a')$ is a direct summand of $\sh O_{\bb P^1_k}(a)$, we have $\max_j -a'_j \leq \max_i -a_i$.
  Moreover, as $\sh O_{\bb P^1_k}(a) \otimes_{\sh O_{\bb P^1_k}} \sh O_{\bb P^1_k}(a')$ surjects onto $\sh O_{\bb P^1_k}(a'')$, it follows that $\max_k -a''_k \leq \max_{i,j} -a_i-a'_j \leq 2 \max_i -a_i$.
  Since $\log_2 [\sh A:U]$ such base changes suffice for the construction of a finite locally free $X$-scheme of which the normalisation is $\nc{Y}$, it follows that for the type $b$ of $\nc{Y}$, its length $t$ is at most $[X:\bb P^1_k - S_\infty][\sh A:U]^{\log_2 [\sh A:U]}$, and $\sum_j -b_j \leq t \max_j b_j \leq t [\sh A:U] \gamma_{\sh A}$.

  The result now follows from \autoref{d:module-cohomology}.
\end{proof}

\subsection{Application to computation for constructible sheaves}

In this section, we will indicate how to compute $\hl^1(X_{k^\sep,\et},\sh A)$ for $X$ a smooth connected curve and $\sh A$ an arbitrary constructible sheaf of abelian groups, of torsion invertible in $k$, under the following assumptions.
We will assume a presentation of constructible sheaves (and morphisms between them) to be given, with respect to which one can perform certain operations.
These operations are:
\begin{itemize}
  \item one can compute finite direct sums of constructible sheaves;
  \item one can compute kernels and cokernels of morphisms;
  \item for a constructible sheaf $\sh A$, one can compute a non-empty open subscheme $U$ of $X$ such that $\sh A|_U$ is finite locally constant;
  \item for a closed immersion $i \colon Z \to X$ and its open complement $j \colon U \to X$, one can compute the functors $i^{-1}, i_*, i^!, j_!, j^{-1}, j_*$ and the corresponding units and counits of adjunction; given $\sh A_Z$ on $Z_{\et}$, $\sh A_U$ on $U_{\et}$, and $\phi \colon \sh A_Z \to i_* j^{-1} \sh A_U$, one can compute the corresponding constructible sheaf on $X_{\et}$.
\end{itemize}

In theory, one should be able to give such a presentation using recollement (as done in \autoref{h:torsors} for $j_!\sh A$-torsors), but we will not work this out in this paper.

So suppose $X$ is a smooth connected curve, and $\sh A$ is a constructible sheaf of abelian groups on $X_{\et}$, of torsion invertible in $k$.
Let $U$ be a non-empty open subscheme for which $\sh A|_U$ is finite locally constant.
Use \autoref{a:reduction-separator} and \autoref{a:riemann-roch} to find $l/k$ finite purely inseparable, $V \subs U_l$ open and a finite locally free morphism $\nc{X}_l \to \bb P^1_l$ such that $V$ and $X_l$ are finite locally free over their images in $\bb P^1_l$.
Write $j \colon V \to X_l$ for the inclusion, and write $i \colon Z \to X_l$ for its closed complement.

Compute the canonical short exact sequence $0 \to j_!j^{-1} \sh A \to \sh A \to i_*i^{-1} \sh A \to 0$, and the morphism $\delta(i,j) \colon \hl^0(X_{k^\sep,\et},i_*i^{-1} \sh A) \to \hl^1(X_{k^\sep,\et},j_!j^{-1}\sh A)$ which sends a section of $i^{-1} \sh A$ to the constructible sheaf defined by $j^{-1} \sh A$ on $V_{\et}$, $0$ on $Z_{\et}$, and the section of $i^{-1}j_*j^{-1} \sh A$ obtained from the given section of $i^{-1} \sh A$ by composition with $i^{-1}\sh A \to i^{-1}j_*j^{-1}\sh A$.
Then, as $\hl^1(X_{k^\sep,\et},i_*i^{-1}\sh A) = 0$, we see that $\hl^1(X_{k^\sep,\et},\sh A) = \coker \delta(i,j)$.

This is independent of the choice of $(i,j)$ in the following sense.
If $j' \colon V' \to X_l$ and $i' \colon Z' \to X_l$ are given, with $V' \subs V$ and a given morphism $Z \to Z'$ over $X_l$, then we can compute a commutative diagram
\[
  \begin{tikzcd}
    0 \ar[r] & j'_!(j')^{-1} \sh A \ar[r] \ar[d] & \sh A \ar[r] \ar[d,equals] & i'_*(i')^{-1} \sh A \ar[r] \ar[d] & 0 \\
    0 \ar[r] & j_!j^{-1} \sh A \ar[r] & \sh A \ar[r] & i_*i^{-1} \sh A \ar[r] & 0
  \end{tikzcd}
\]
a commutative diagram
\[
  \begin{tikzcd}
    \hl^0(X_{k^\sep,\et},i'_*(i')^{-1} \sh A) \ar[r] \ar[d] & \hl^1(X_{k^\sep,\et},j'_!(j')^{-1} \sh A) \ar[d] \\
    \hl^0(X_{k^\sep,\et},i_*i^{-1} \sh A) \ar[r] & \hl^1(X_{k^\sep,\et},j_!j^{-1} \sh A)
  \end{tikzcd}
\]
and therefore the morphism $\coker \delta(i',j') \to \coker \delta(i,j)$ corresponding to the identity map on $\hl^1(X_{k^\sep,\et},\sh A)$.

In the same way, we see that for constructible sheaves $\sh A$, $\sh B$ on $X_{\et}$, and a morphism $\sh A \to \sh B$, we can compute the induced morphism $\hl^1(X_{k^\sep,\et},\sh A) \to \hl^1(X_{k^\sep,\et},\sh B)$, and that for a morphism $f \colon Y \to X$ of smooth connected curves, we can compute the pullback $\hl^1(X_{k^\sep,\et},\sh A) \to \hl^1(Y_{k^\sep,\et},f^{-1}\sh A)$.

\bibliographystyle{abbrvnat}
\bibliography{ECC}

\end{document}